\documentclass[12pt]{amsart}
\usepackage{amsfonts,amsthm,xspace}

\DeclareRobustCommand{\SkipTocEntry}[5]{}

\usepackage{amsmath,amssymb}
\usepackage{mathptmx}
\usepackage{graphicx}
\usepackage{geometry}
\usepackage{epstopdf}
\usepackage[linktocpage=true]{hyperref}
\usepackage[usenames,dvipsnames]{color}

\usepackage{stmaryrd}

\newtheorem{theorem}{Theorem}
\newtheorem{proposition}{Proposition}
\newtheorem{remark}{Remark}
\newtheorem{corollary}{Corollary}
\newtheorem{lemma}{Lemma}
\newcommand{\be}{\begin{equation}}
\newcommand{\ee}{\end{equation}}
\newcommand{\Z}{{\mathbb Z}}
\renewcommand{\k}{\mathbf{k}}
\newcommand{\R}{{\mathbb R}}

\newcommand{\C}{{\mathbb C}}
\newcommand{\y}{\mathbf{y}}
\newcommand{\T}{{\mathcal T}}
\newcommand{\G}{{\mathcal G}}

\newcommand{\old}[1]{}

\newcommand{\x}{\mathbf{x}}
\renewcommand{\P}{\mathbb{P}}
\newcommand{\E}{\mathbb{E}}

\renewcommand{\H}{{\mathbb H}}

\newcommand{\cov}{\mathrm{Cov}}
\newcommand{\var}{\mathrm{Var}}
\newcommand{\p}{\mathcal{P}}

\begin{document}

\title{Transfer current and pattern fields in spanning trees}

\author{Adrien Kassel}
\address{Department Mathematik, ETH, R\"{a}mistrasse 101, 8092 Z\"{u}rich, Switzerland}
\email{adrien.kassel@math.ethz.ch}

\author{Wei Wu}
\address{Division of Applied Mathematics, Brown University, 182 George Street, Providence, RI 02912, USA}
\curraddr{Department of Mathematics, New York University, 251 Mercer Street, New York, NY 10012, USA}
\email{weiwu@cims.nyu.edu}



\begin{abstract}
When a simply connected domain $D\subset{\mathbb{R}}^d$ ($d\ge 2$) is
approximated in a ``good'' way by embedded connected weighted graphs, we
prove that the transfer current matrix (defined on the edges of the graph
viewed as an electrical network) converges, up to a local weight factor, to
the differential of Green's function on $D$.

This observation implies that properly rescaled correlations of the spanning
tree model and correlations of minimal subconfigurations in the abelian
sandpile model have a universal and conformally covariant limit.

We further show that, on a periodic approximation of the domain, all pattern
fields of the spanning tree model, as well as the minimal-pattern (e.g.
zero-height) fields of the sandpile, converge weakly in distribution to
Gaussian white noise.
\end{abstract}

\maketitle

\tableofcontents

\section{Introduction}

Let ${\mathcal{G}}$ be a locally finite connected weighted graph, with
weight function $c$ (a positive symmetric function over directed edges). It
represents both an electrical network with conductances $c(e)$ on each bond $%
e$ and a (time-homogeneous) random walk $X$ on the vertex set of ${\mathcal{G%
}}$ with transition probability $\P (X_{n+1}=y\vert X_{n}=x)=c(xy)/\deg_c(x)$%
, where $\deg_c(x)=\sum_{y\sim x}c(xy)$ is the weighted degree of $x$. If $%
e=xy$ is a directed edge and a battery imposes a unit current to flow into $x
$ and out of $y$, the value of the current observed through any bond~$f$ is
the transfer current $T(e,f)$ between $e$ and~$f$.

It is well-known that the transfer current $T(e,f)$ between two directed
edges~$e$ and~$f$ is equal to the algebraic number of visits of the random
walk started at $x$ and stopped when it first hits $y$, see e.g.~\cite%
{DS,BP,BLPS}. This relates the transfer current to Green's function $G$ for
the random walk (defined in Section~\ref{dha}), and we have 
\begin{equation*}
T(ab,uv)=c(uv)(G(a,u)-G(b,u)-G(a,v)+G(b,v))\,.
\end{equation*}
Since Green's function is symmetric, the matrix $K(e,e^{\prime })=\sqrt{%
c(e)/c(e^{\prime })}T(e,e^{\prime })$\label{symkern} is symmetric in both
arguments. This symmetry is called the reciprocity law in electrical theory.
Further relations between electrical quantities and random walk hitting
times and commute times are well-known, see e.g.~\cite{DS,Lo} and references
therein.

One can give an expression for $T$ involving an eigenbasis of the Laplacian~$%
\Delta$. Let $(f_k)_{k\ge 1}$ be an orthonormal basis of eigenvectors for $%
\Delta$ associated to eigenvalues $\lambda_k\neq 0$. Then, for any two edges 
$ab,uv\in E$, we have 
\begin{equation*}
T(ab,uv)=c(uv)\sum_{k\ge 1}\frac{(f_{k}(u)-f_{k}(v))(f_{k}(a)-f_{k}(b))}{%
\lambda_{k}} \,.  \label{TDelta}
\end{equation*}
In the case of infinite periodic graphs, the transfer current may be
evaluated explicitly by this means.

Because of its relationship to random walk, the transfer current appears in
the study of many probabilistic models on graphs. The random spanning tree
model~${\mathcal{T}}$ is the probability measure on spanning trees of ${%
\mathcal{G}}$ which assigns a probability proportional to the products of
the edge-weights of the tree (see Figure~\ref{UST} for a sample). 
\begin{figure}[ht]
\centering
\includegraphics[width=10cm,height=9cm]{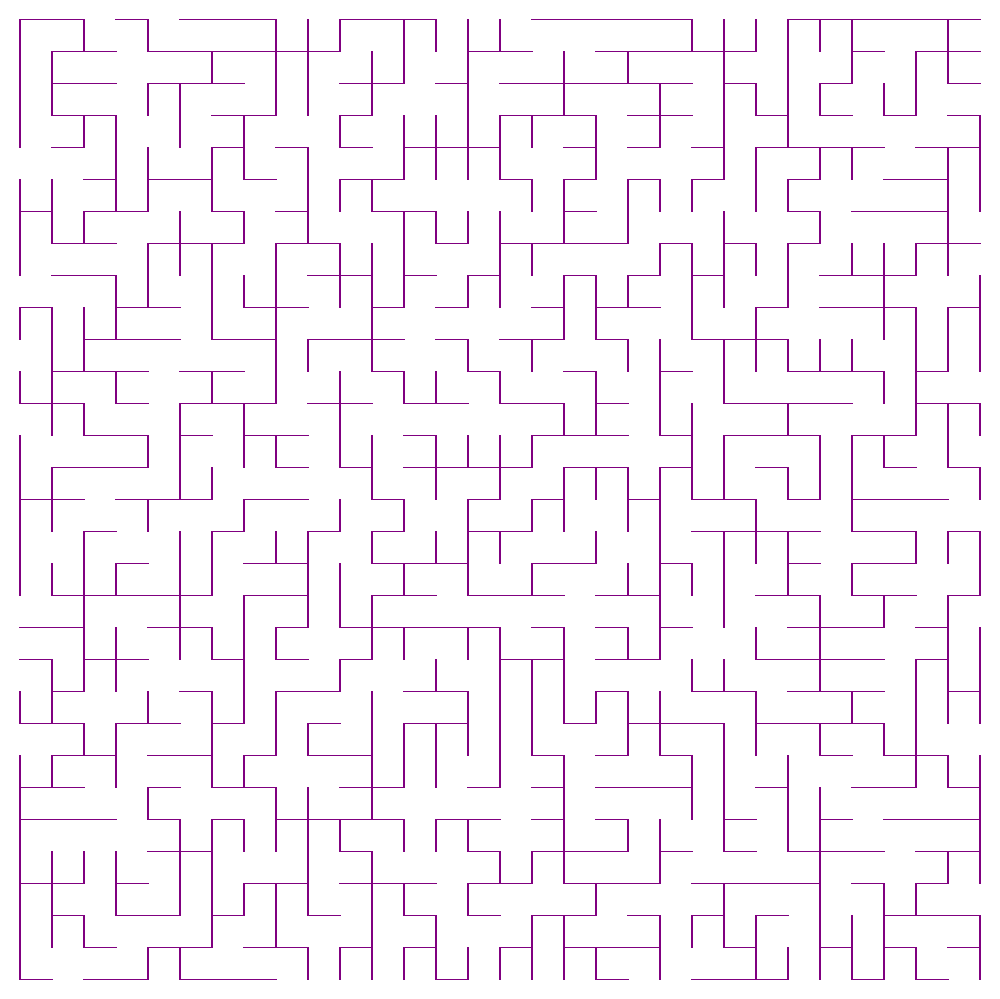}
\caption{A (free boundary) uniform spanning tree on a $30\times 30$ square
grid}
\label{UST}
\end{figure}

The transfer current theorem of Burton and Pemantle~\cite{BP} states that
the spanning tree model, viewed as a point process on the set of undirected
edges of the graph, is determinantal with kernel $T$. This means that all
local statistics have a closed-form expression in terms of the transfer
current of the local edges: for any finite collection of disjoint undirected
edges $e_1,\ldots,e_k$ in the graph, we have 
\begin{equation*}
\P \left(e_1,\ldots,e_k\in \mathcal{T}\right)=\det
\left(T\left(e_i,e_j\right)\right)_{1\leq i,j\leq k}\,.
\end{equation*}
Any local statistics of the spanning tree is thus a local computation
provided we know the value of the transfer current (which depends on the
graph in its whole). This is the case for periodic infinite graphs as
already pointed out above. Note that the matrix $K$ defined above is also a
kernel of this process (since it is the conjugate of $T$ by a diagonal
matrix): since it is symmetric, we say that the spanning tree model is a
symmetric determinantal process (this implies further properties of the
process that we review in Section~\ref{patternfieldsust}).

In many cases, although an explicit form of the transfer current is not easy
to obtain, one can show that on a large scale, or equivalently in the
scaling limit for macroscopically distant points, the transfer current is
close to its continuous counterpart. For graphs which approximate a smooth
domain in a nice way (we call these good approximations), the transfer
current between two macroscopically distant edges tends (up to a local
constant depending on the edge weights) to the derivative of Green's
function along the directions given by the edges. This is our first main
result (see Theorem~\ref{ticv}). The proof is an adaptation of the classical
arguments from a special case of~\cite{CS}. The goal is to show the
universality of the limit on a large class of graphs in any dimension. Our
assumptions on the graph are quite strong (thus the proof is simpler), but
these are satisfied by many interesting examples.

Three other models are closely related to the spanning tree model: the
two-component spanning forest, the spanning unicycle, and the abelian
sandpile model of Dhar~\cite{D}. There exist natural couplings between them
which enable to express the probability of any event for these models as a
spanning tree computation (see~\cite{KKW} for the first two models and see~%
\cite{MD,JW} for the abelian sandpile; also see~\cite{LP}). However, local
statistics for these models are not necessarily local statistics for the
spanning tree and may require the knowledge of the geometry of the spanning
tree as a whole. We focus in this paper only on the local statistics that
are local statistics for the spanning tree as well. These type of statistics
can be computed explicitly and have a well-defined universal limit, expressed in terms of the scaling limit of transfer currents. This is our second main
result (Theorems~\ref{corrUST} and~\ref{corrASM}; another application is
Theorem~\ref{dgff}).

A possible realization of the random process in a finite subregion is called
a pattern. Given a collection of countable patterns on disjoint supports, we
can define a new point process, or equivalently a random spin field which is
the indicator of the occurrence of these patterns. Under sufficiently fast
decay of correlations with the distance, pattern fields of determinantal
point processes converge to Gaussian white noise (Theorem~\ref{cltdet}). We
use this to study the fields generated by the local events mentioned above
in the spanning tree and sandpile models. More precisely, we express local
events fields as pattern fields in the spanning tree model and use the
determinantal structure of the spanning tree process to show that these
fields are Gaussian white noise. This is our third main result (Theorems~\ref%
{patterncv},~\ref{USTpattern}, and~\ref{ASMpattern}). For the special case
of the zero-height field of the sandpile on~${\mathbb{Z}}^{2}$, the result was
already obtained by D\"{u}rre. The study of pattern fields of determinantal
point processes originated in the work of Boutillier~\cite{Bo}, who studied
them in fluid dimer models on the plane. Gaussian fluctuations for symmetric
determinantal point fields were first observed by Soshnikov~\cite{So} (this
corresponds to patterns consisting of a single point present).

Computing probabilities of local events which are not local for the spanning
tree is much harder. In the planar case, it has been addressed by Wilson~%
\cite{Wilson} using methods of Kenyon and Wilson~\cite{KenyonWilson}.

Describing the scaling limits of spanning trees and related models is
challenging. In two dimensions, using conformal invariance, the scaling
limit is well understood in terms of Schramm--Loewner evolutions~\cite{LSW}.
In dimension three, Kozma~\cite{Kozma} has shown the existence of the
scaling limit of the loop-erased random walk (that is, the branches of the
spanning tree by a theorem of Pemantle~\cite{Pemantle}) and in high dimension, the scaling
limit is also known due to the lace expansion~\cite{Slade}. Furthermore, the
geometry of the infinite volume limit has been well-studied~\cite{BLPS,BKPS}%
. However, many natural questions about the geometry of spanning trees on
infinite graphs and their scaling limits remain open. The pattern fields we
study encode certain information of spanning trees in any dimension, and
their scaling limit might tell us something about the scaling limit of
spanning trees. Our method is based on the determinantal nature of spanning
trees and the study of its kernel, the transfer current. Pattern fields are
local quantities, however we give in Section~\ref{treescrst}, for future
use, some properties which relate the transfer current to global geometric
properties of trees.

The paper is organized as follows. In Section~\ref{dha}, we introduce some
concepts from discrete harmonic analysis: weighted graphs, vector field,
derivative, Laplacian, harmonic functions, Green's functions, and transfer
current. In Section~\ref{sl} we define and exhibit good approximations for
domains in ${\mathbb{R}}^{d}$. We prove Theorem~\ref{ticv} which states that
on good approximations, the transfer current convergences (up to a local
factor) to the derivative (in both variables along the direction of the
edges) of Green's function. In Section~\ref{i} we review some couplings
between spanning trees and the other models and relate the probability of
events that correspond to local events for the spanning tree, expressed in
terms of the transfer current. We derive from this two types of results: on
the one hand, the connection between combinatorics of trees and discrete
analytic properties of $T$; on the other hand, we use Theorem~\ref{ticv} to show
that certain correlations of the trees and sandpiles have a universal and
conformally covariant limit. Finally, we concentrate in Section~\ref{dpp} on
pattern fields of certain determinantal point processes and show as an
application of what precedes (for this, we only need the order of decay of $%
T(x,y)$ in the limit $|x-y|\to\infty$, not its precise value) that the
pattern fields of a random spanning tree and the minimal-pattern (e.g.
zero-height) fields of the abelian sandpile converge to Gaussian white noise
in the scaling limit, under good approximation of a domain.

\section{Discrete harmonic analysis}

\label{dha}

Let ${\mathcal{G}}=\left(V,E,c,\partial V\right)$ be a finite connected
weighted graph with vertex-set $V$, edge-set~$E$, weight function $c:E\to{%
\mathbb{R}}_+^*$, and a boundary-vertex-set $\partial V$ which consists of a
(possibly empty) subset of the vertices. We let $V^{\circ}=V\setminus%
\partial V$ and call its elements the interior vertices.

Let $\Omega ^{0}$ denote the space of functions over the interior vertices
and $\Omega ^{1}$ be the space of \mbox{$1$-forms}, that is antisymmetric
functions over the set~$E^{\pm }$ of directed edges (two directed edges per
element of $E$). We endow $\Omega ^{0}$ with its canonical scalar product
coming from its identification with ${\mathbb{R}}^{V^{\circ }}$, that is $%
\left\langle f,g\right\rangle =\sum_{v\in V^{\circ }}f(v)g(v)$, and endow $%
\Omega ^{1}$ with the scalar product given by $\left\langle \alpha ,\beta
\right\rangle =1/2\sum_{e\in E^{\pm }}c(e)\alpha (e)\beta (e)$.

In the case where ${\mathcal{G}}$ is infinite, we now abuse the notations $%
\Omega^0$ and $\Omega^1$ and suppose that these actually denote the
subspaces of elements with finite norm (the $\ell^2$ spaces).

We define the map $d:\Omega^0\to\Omega^1$ by $df(vv^{\prime })=f(v^{\prime
})-f(v)$ and the map $d^*:\Omega^1\to\Omega^0$ by $d^*\alpha(v)=\sum_{v^{%
\prime }\sim v}c(v^{\prime }v)\alpha(v^{\prime }v)$. The maps $d$ and $d^*$
are dual to one another, in the sense that $\left\langle f,
d^*\alpha\right\rangle=\left\langle df,\alpha\right\rangle$ for any $%
f\in\Omega^0$ and $\alpha\in\Omega^1$.

The Laplacian is defined by $\Delta=d^*d:\Omega^0\to\Omega^0$, and it is
easy to check that for any $f\in\Omega^0$ and $v\in V^{\circ}$, we have 
\begin{equation*}
\Delta f(v)=\sum_{v^{\prime }\sim v}c(vv^{\prime })\left(f(v)-f(v^{\prime
})\right)\,,
\end{equation*}
where the sum is over neighboring vertices (including boundary vertices). A
function $f$ is said to be harmonic at a vertex $v$ whenever $\Delta f(v)=0$.

The Green function $G$ is the inverse of the Laplacian. Two cases may occur:
if the boundary is non empty ($\partial V\ne \emptyset$), $\Delta$ is
invertible on $\Omega^0$, hence $G=\Delta^{-1}$. In the case when the
boundary is empty ($\partial V=\emptyset$), $G$ is the inverse of $\Delta$
on the space of mean zero functions. The first one is called the Dirichlet
Green function (or wired Green function), the second is called the Neumann
Green function (or free Green function)~\footnote{%
Our choice of the qualifier ``Neumann'' for this Green function comes from
the fact that we can replace the free boundary conditions by Neumann
boundary conditions by artificially adding a boundary to the graph. The
Laplacian acts on the extended space of functions which take same values at
any boundary vertex and its adjacent interior vertex. This corresponds to a
vanishing normal derivative on the boundary. This modification of the
Laplacian does not change the local times of the random walk at interior
points which is sufficient for our needs. In the limit, this ``Neumann''
Green function actually converges to the elementary solution for the Neumann
problem.}. Given a weighted graph~${\mathcal{G}}$ we will sometimes consider
simultaneously both Green's functions by looking at the wired or free
boundary conditions.

When the space of harmonic forms, $H=\mathrm{Im}(d)\cap \mathrm{Ker}(d^*)$
is trivial (which is always the case on a finite connected graph: harmonic
functions are constant), the space of~$1$-forms has an orthogonal
decomposition (Hodge) as 
\begin{equation*}
\Omega^1=\mathrm{Im}(d)\oplus \mathrm{Ker}(d^*)\,,
\end{equation*}
Under this assumption ($H={0}$), when the graph is infinite, Theorem 7.3. of 
\cite{BLPS} implies that free and wired spanning forests measures coincide.

The orthogonal projection of $\Omega^1$ onto $\mathrm{Im}(d)$ is $P=dGd^*$.
Any directed edge $xy$ defines a $1$-form $\delta_{xy}-\delta_{yx}$ and we
often identify the directed edge with this $1$-form. The \emph{transfer
current} is the matrix of $P$ in the basis indexed by~$E^{\pm}$. For two
directed edges $xy$ and $uv$, it therefore satisfies 
\begin{align}\label{ti}
\begin{split}
T(xy,uv)&=\sqrt{c(xy)}^{-1}\left\langle xy/\sqrt{c(xy)},P uv/\sqrt{c(uv)}%
\right\rangle \sqrt{c(uv)}=c(xy)^{-1}\left\langle xy,P uv\right\rangle \\
&=c(uv)\left(G(x,u)-G(y,u)-G(x,v)+G(y,v)\right)\,.
\end{split}
\end{align}
According to the boundary condition chosen (free or wired), we obtain two
different transfer currents.

A vector field on ${\mathcal{G}}$ is the choice for each vertex $v\in
V^{\circ}$ of an outgoing edge $X_v$. For any vertex $v$ and directed edge $%
vv^{\prime }$, we define $v+(vv^{\prime })$ to be $v^{\prime }$. In
particular, $v+X_v$ is the neighbor of $v$ to which the vector $X_v$ points
at. We define the derivative of a function~$f$ with respect to the vector
field $X$ to be the function $v\mapsto \nabla_X f(v)=f(v+X_v)-f(v)$. The
derivative of a function $f$ with respect to a vector field $X$ satisfies $%
\nabla_X f=df(X)$, an equality which is coherent with the one from calculus.

The $1$-forms are dual to vector fields and naturally act on them. Another
way to see the action of $P$ is 
\begin{equation*}
\langle\alpha,P\beta \rangle=\langle d^* \alpha, G d^*\beta \rangle\,.
\end{equation*}
Under this formulation, the projection can be written as acting on the
divergence of vector fields with kernel given by Green's function. The
continuum analog of this projection is widely used and is known as the
Hodge-Helmholtz projection.

\section{Scaling limits}

\label{sl}

\subsection{Good approximations of domains}

Let $d\ge 2$ and $D\subset {\mathbb{R}}^d$ be a simply connected open set,
with smooth boundary $\partial D$. We call such a set $D$ a domain. We
consider a sequence of connected weighted graphs with boundary-set ${%
\mathcal{G}}_n=(V_n,E_n,c_n,\partial V_n)$ embedded in $D$. By ``embedded'',
we mean that the vertices are points inside $D$ and boundary vertices lie in 
$\partial D$. The edges are mapped to smooth non-intersecting segments in
such a way that edges between boundary vertices lie inside $\partial D$. We
let $\Delta_n$ be the Laplacian on ${\mathcal{G}}_n$.

We say that a sequence of functions $f_{n}:V_{n}^{\circ }\rightarrow {%
\mathbb{R}}$, converges uniformly on a compact subset of $K\subset D$ to a
function $f$, if the sequence $(g_{n})_{n\geq 1}$ defined by $%
g_{n}=\sum_{x\in V_{n}^{\circ }}f_{n}(x)\delta _{x}$ converges uniformly to~$%
f$ on $K$.

We say that the sequence $({\mathcal{G}}_{n})_{n\geq 1}$ is a \emph{good
approximation} of $D$ if the following properties hold.

\begin{enumerate}
\item (\emph{Approximate mean value property}) For any bounded harmonic
function $f$ on $V_n$, and any ball $B(v,r)$, we have $f(v)-1/\vert
B(v,r)\vert\sum_{w\in B(v,r)}f(w)=O(1/r)$.

\item (\emph{Paths approximation}) For any straight line $\gamma $ in $D$
joining $x$ to $y$, there exists finite paths $\gamma_n$ inside ${\mathcal{G}%
}_{n}$ joining two vertices $x_{n}$ to $y_{n}$ such that $x_{n}\rightarrow x$%
, $y_{n}\rightarrow y$, and $\gamma _{n}$ uniformly tends to a path $\gamma $%
. Moreover, the discrete length of $\gamma _{n}$ is bounded by an absolute
constant times the length of $\gamma $.
\end{enumerate}

Examples of good approximations include: the lattices ${\mathbb{Z}}^{d}$ for 
${\mathbb{R}}^d$ and ${\mathcal{G}}_n={\mathbb{Z}}^d/n\cap D$ for some
domain $D$ (the Approximate mean value property follows from Lemma~6 in~\cite%
{JLS}; the Paths approximation property is clear), and for $d=2$, isoradial
graphs with bounded half-angles, see~\cite{Ken02} (see Appendix~\ref{iso}
for a quick review of the definition of isoradial graphs and a justification
of why there are good approximations).

The Paths approximation property is rather easy to check on a given graph.
It is less so for the Approximate mean value property: see Remark~\ref%
{suffcond} below for a practical way to check it.

\subsection{Convergence of the derivative}

In the following, $|X(v)|$ denotes the Euclidean length of the edge viewed
as an embedded segment. For any vertex $v$, we denote $B(v,r)$ the discrete
ball of center $v$ and radius $r$ in~${\mathcal{G}}$, and by $|B(v,r)|$ its
cardinality.

We say that a compact subset $K$ of $D$ is interior if its distance to the
complement of~$D$ is strictly positive.

\begin{proposition}
\label{cv} Let $({\mathcal{G}}_n)_{n\ge 1}$ be a good approximation of a
domain $D$. Consider a vector field $X_n$. Suppose a sequence of harmonic functions on ${\mathcal{G}}%
_n$ converges uniformly on interior compact subsets of $D$ to a harmonic
function $f$. Then the sequence of their rescaled discrete derivatives $%
|X_n|^{-1}\nabla_{X_n}f$ is uniformly close in the limit $n\to\infty$, on all interior compact subsets of $D$, to
the derivative $\nabla_{X_n}f$ of $f$.
\end{proposition}

\begin{proof}
Let $K$ be a compact subset of $D$. Let $K'$ be an interior compact subset such that $K\subset K'\subset D$ and $K$ is at a positive distance $r$ from the exterior of $K'$. The function $f$ is bounded on $K'$, hence the sequence $(f_n)_{n\ge 1}$ as well. Let $z\in K$ and $v$ be an approximation of $z$. By the approximate mean value property at $v$ and $v+X_v$, there exists a constant $C$ such that, for any $n\ge 1$, we have
\[
\left\vert\nabla_{X_n}f_n(v)-\frac{1}{\vert B(v,nr)\vert}\sum_{w\in V_n^{\circ}}f_n(w)\left(1_{\{w\in B(v+1,nr)\}}-1_{\{w\in B(v,nr)\}}\right)\right\vert \le \frac{C}{nr}\,.
\]
Now, 
\begin{align*}
\sum_{w\in V_n^{\circ}}f_n(w)\left(1_{\{w\in B(v+1,nr)\}}-1_{\{w\in B(v,nr)\}}\right)
&\le C \left\vert B\left(v,(n+1)r\right)\setminus B(v,nr)\right\vert\\
&\le C\frac{\vert\partial B\left(v,(n+1)r\right)\vert}{\vert B(nr)\vert}=O(C/nr)\,,
\end{align*}
where we used, in the last equality, the bound of $O(1/R)$ on the surface-area-to-volume ratio of a ball or radius $R$ in $\R^d$.
The sequence $|X_n|^{-1}\nabla_{X_n}f_n(v)$ is therefore uniformly bounded on $K$.

Therefore, one can extract a subsequence $\left\{ n_{k}\right\} $, so that $%
|X_{n_{k}}|^{-1}\nabla _{X_{n_{k}}}f_{n_{k}}$ converges uniformly on compact subsets to
some bounded function $h:D\rightarrow \mathbb{R}$. We now prove that for any
subsequential limit $h=\nabla _{X_{n_k}}f$, which finishes the proof.

By the \emph{Paths approximation} property (and because $K$ is locally convex), for any $u,v\in K$ close enough, and any $n$, we
can take a discrete poly-line segment $\overline{u_{n}v_{n}}$, such that $%
u_{n}\rightarrow u$, $v_{n}\rightarrow v$, and $\overline{u_{n}v_{n}}$
converges (in the uniform topology) to the line $\overline{uv}$. We replace the vector field along this line so that it is aligned with it. Along a
subsequence $\left\{ n_{k}\right\} $, we have 
\begin{equation*}
f_{n_{k}}\left( v_{n_{k}}\right) -f_{n_{k}}\left( u_{n_{k}}\right) =\sum_{s\in \overline{u_{n_{k}}v_{n_{k}}}\cap V_{n}}df_{n_{k}}\left(
{X_{n_{k}}}\right) =\int_{\overline{u_{n_{k}}v_{n_{k}}}}|X_{n_k}|^{-1}\nabla
_{X_{n_{k}}}f_{n_{k}}\left( x\right) dx\,.
\end{equation*}%
Taking $n_{k}\rightarrow \infty $, uniform convergence implies
\begin{equation*}
f\left( v\right) -f\left( u\right) =\int_{\overline{uv}}h\left( x\right) dx\,.
\end{equation*}%
By taking $\overline{uv}=X_{u}$ at every $u$, and letting $v\rightarrow u$, we have $h=\nabla _{X}f$.
\end{proof}

\begin{remark}
Since the derivative of a harmonic function is a harmonic function, Proposition~\ref{cv} may be applied successively to show that any higher order
derivative also converges under the same assumptions.
\end{remark}

\subsection{Green's function and transfer current}

In this paper, we consider $\Delta=-\sum_{i=1}^d{\partial^2}/{\partial x_i^2}
$ to be the positive definite Laplacian. Recall that the Green function with
Neumann (resp. Dirichlet) boundary conditions is the (unique up to constant)
smooth symmetric kernel over $D$ solution to 
\begin{equation*}
\Delta_y u(x,y)=\delta_{x}
\end{equation*}
with boundary condition $\partial u(x,y)/\partial n(y)=0$ for $y\in\partial D
$, where $n(y)$ is the normal vector at the boundary point $y$ (resp. $%
u(x,y)=0$ for $y\in\partial D$). On a bounded domain the Neumann Green
function is defined up to an additive constant, and we identify it with its
action on functions of mean zero.

We now make some further assumptions on our good approximations. First, we
suppose that they form an exhausting sequence of some infinite embedded
graph ${\mathcal{G}}_\infty$. We assume that Green's function on ${\mathcal{G%
}}_\infty$, denoted by~$G_0$, converges uniformly, upon rescaling, on
compact sets (away from the diagonal) to the continuous Green's function $g$
on ${\mathbb{R}}^{d}$ with control on the error term, and that the
invariance principle holds:

\begin{enumerate}
\item[A1] \label{assumption4}(\emph{Green's function asymptotics}) $%
G(z,w)=n^{2-d}g(z,w)+O\left(n^{-d}|z-w|^{-d}\right) $.

\item[A2] (\emph{Invariance principle}) The random walk on $({\mathcal{G}}%
_{n})_{n\geq 1}$ converges in the supremum norm to Brownian motion in $D$.
\end{enumerate}

These assumptions are satisfied by the following good approximations: the
lattices ${\mathbb{Z}}^d$, see~\cite{FU,LL}, and for $d=2$, isoradial
graphs, see~\cite{Ken02,CS} and Appendix~\ref{iso}.

\begin{theorem}
\label{ticv} Suppose the conditions of Proposition~\ref{cv} hold. Then under
assumptions (A1) and (A2), we have 
\begin{equation}
T(e_{n},e_{n}')=c(e'_n)n^{-d}dg_{D}\vert_{(z,w)}(e,e')+o\left( n^{-d}|z-w|^{-d}\right) \,, \label{main}
\end{equation}
where $g_D$ is Green's function with the corresponding boundary conditions on $D$, and $dg_D$ its differential.
\end{theorem}

\begin{proof}
Let $T_{0}$ denote the transfer current associated with Green's function on
the whole space, defined in~\eqref{ti}. By Assumption A1, the rescaled
Green's function $n^{d-2}G_{0}$ on the infinite graph converges uniformly on
all interior compact subsets to Green's function on the whole plane $g$.
Applying Assumption A1 to~\eqref{ti} and differentiating $g$ twice, we have 
\begin{equation}
T_{0}(e_{n},e_{n}^{\prime })=c(e_{n}^{\prime })|e_{n}||e_{n}^{\prime
}|n^{-d}\nabla _{e}^{z}\nabla _{e^{\prime }}^{w}g(z,w)+o\left(
n^{-d}|z-w|^{-d}\right) \,.  \label{t0cv}
\end{equation}%
Consider the discrete harmonic function $F=G-G_{0}$. By the invariance
principle (Assumption A2), $n^{d-2}F$ converges to the continuous harmonic function $f$ with
boundary values given by $g_{D}-g$. Applying the same (mean value) argument
as in the proof of Proposition \ref{cv}, the discrete gradient of $F$ is uniformly bounded.
Therefore $F$ converges to~$f$, uniformly on interior compact subsets of~$D$. By
applying Proposition~\ref{cv} twice, we obtain the uniform convergence of
the rescaled double increment of $F$ to the double derivative of~$f$. Since 
\[
T(e_{n},e_{n}^{\prime })=T_{0}(e_{n},e_{n}^{\prime })+c(e_{n}^{\prime
})\left( F\left( z,w\right) -F\left( z+e_{n},w\right) -F\left(
z,w+e_{n}^{^{\prime }}\right) -F\left( z+e_{n},w+e_{n}^{^{\prime }}\right)
\right) ,
\]%
combined with (\ref{t0cv}), this implies 
\[
T(e_{n},e_{n}^{\prime })=c(e_{n}^{\prime })|e_{n}||e_{n}^{\prime
}|n^{-d}\nabla _{e}^{z}\nabla _{e^{\prime }}^{w}\left( g+f\right)
(z,w)+o\left( n^{-d}|z-w|^{-d}\right) \,,
\]%
which finishes the proof. 
\end{proof}

\begin{remark}
When the domain is $D={\mathbb{R}}^d$, and we dispose of a two-terms
expansion of Green's function (as is the case for ${\mathbb{Z}}^d$, by~%
\cite{FU}), the proof of Theorem~\ref{ticv} is immediate by computation and
Taylor expansion.
\end{remark}

\begin{remark}
By using the interpretation of Green's function as the density of time spent in the neighborhood of $y$ when started at $x$ and the explicit Brownian time-space scaling, one can use (A2) to prove a weaker form of (A1), namely that $G(z,w)=n^{2-d}g(z,w)+o(n^{2-d}|z-w|^{2-d})$.
\end{remark}

\begin{remark}\label{suffcond}
If we suppose that the discrete Laplacian of homogeneous polynomials of degree~$2$ in $\R[x_1,\ldots,x_d]$ is $2\sum_{i=1}a_i$, where $a_i$ is the coefficient of $x_i^2$, then, by using (A.1), the Approximate mean value property may be shown, following the proof of Proposition~A.2 in~\cite{CS}. This assumption also implies that random walk is isotropic (and the variance of the increments is of order of the volume around each vertex) and thus implies~(A.2).
\end{remark}

As in the setup of~\cite[Theorem~13]{Ken00}, in order to control the behavior of Green's function for points on the boundary, we slightly modify the way we approximate the domain $D$ by ensuring that the approximating graphs have piecewise linear boundaries in the following sense. We consider a sequence $\delta=\delta(n)$ such that $n^{-1}\delta^{-1}=o(1)$, as $n\to\infty$, and an increasing sequence of domains $D^\delta\subset D$ such that $D^\delta$ lies within $\delta$ of $D$ and such that its boundary is a polytope (polygon, when $d=2$) with hyperfaces (segments, when $d=2$) of size $\delta^{d-1}$. $D^\delta$ is furthermore assumed to be convex for $d\ge 3$.

We can obtain convergence of the transfer current for points on the boundary whenever we consider, as a replacement of a good approximation to $D$, a diagonal subsequence of good approximations of the domains $D^\delta$. This follows from the fact that on domains with piecewise linear boundaries, we can use a reflection argument.

\begin{corollary}
Let $D\subset \mathbb{R}^{d}$, $d\geq 3$ be a convex domain with smooth
boundary, or $D\subset \mathbb{R}^{2}$ be simply connected with smooth
boundary. Assume the approximation sequence is chosen as above. Formula \eqref{main} holds for points on the boundary.
\end{corollary}

\begin{proof}
Reflect the approximation graph of $D^\delta$ across the flat piece of the boundary, and
glue it with the original graph. The result follows by noting that Green's function on the new graph is a linear combination of Green's
function on the original graph. When $d=2$ the argument works for any simply
connected surface (see Proposition~\ref{surface}).
\end{proof}

\subsection{Planar graphs}


On any planar embedded graph, the conjugate of a harmonic function $h$ on
the graph is a harmonic function $h^*$ on the planar dual which is defined
by the following discrete Cauchy--Riemann equations: (for any directed edge $%
e$, denote by $e^*$ the dual edge directed in such a way that $e\wedge e^*>0$%
) $dh^*(e^*)=dh(e)$ for any dual edges $e$ and~$e^*$ (where $d$ is the
discrete derivative in the dual and primal graph, respectively).

In particular, we have the following.

\begin{lemma}
\label{harmonicconjugate} Let $G$ be the Green function of a planar graph~${%
\mathcal{G}}$. Let $\tilde{G}$ be the harmonic conjugate of $G$. For any $%
e=vv'$, let $e^*=pp'$ be the dual edge to $e$. We have 
\begin{equation*}
\tilde{G}(p,q)-\tilde{G}(p',q)=G^*(p,q)-G^*(p',q)\,,
\end{equation*}
where $G^*$ is the Green function on the dual graph~${\mathcal{G}}^*$.
\end{lemma}

Lemma~\ref{harmonicconjugate} is the discrete analog to the fact that the
`complex' Green function defined as $g_D+ig_N$ where $g_D$ is the Dirichlet
Green function and $g_N$ the Neumann Green function, is analytic (away from
the diagonal) hence satisfies the Cauchy--Riemann equations.

On the whole plane, we can compute this explicitly for comparison. Start
with $g(z,w)=-1/(2\pi) \log|z-w|$. Then by taking the harmonic conjugate (in
one of the variables) we obtain $-1/(2\pi)\mathrm{arg}(z-w)$ (it is the same
in either variable by a simple geometrical fact about parallel angles being
equal). Now taking the harmonic conjugate with respect to the other variable
we get $-1/(2\pi)\log|z-w|$ back again, up to an additive constant. By
taking differences of Green functions as in the lemma we get equality.

\begin{proof}[Proof of Lemma~\ref{harmonicconjugate}]
Let $f(w)=G(v,w)-G(v',w)$. The function $f$ is harmonic with a defect of harmonicity $+1$ at $v$ and $-1$ at $v'$. Its harmonic conjugate is the function $f^*(w)=\tilde{G}(p,q)-\tilde{G}(p',q)$. 
This is a univalued harmonic function with a defect of harmonicity $+1$ at $p$ and $-1$ at $p'$ (This may be seen by comparing the harmonic extension of $f$ avoiding edge $vv'$ and the harmonic extension crossing that edge. Since this move is local, when one extends from $p$ to $p'$, one obtains the result). Thus it is equal to $G^*(p,q)-G^*(p',q)$. 
\end{proof}

\begin{corollary}
\label{tidual} Let $e$ and $e'$ be two edges. Then 
\begin{equation*}
T(e,e')=T^{\ast }(e^{\ast },e^{\prime \ast })\,,
\end{equation*}%
where $T$ is the transfer current on ${\mathcal{G}}$, and $T^{\ast }$ the
transfer current on the dual graph~${\mathcal{G}}^{\ast }$.
\end{corollary}

A purely combinatorial proof of Corollary~\ref{tidual} follows from
Equations~\eqref{trees} and~\eqref{crst} in Section~\ref{treescrst} below by
noticing that the dual of a two-component spanning forest is a spanning
unicycle.

\section{Transfer current in statistical physics}

\label{i}

In this section, we explain ways in which the transfer current describes
models of statistical physics on weighted graphs: point correlations and
geometrical properties. Using Theorem~\ref{ticv}, we can find the scaling
limits of some of these expressions. We also relate analytical properties of 
$T$ to its combinatorial properties.

\subsection{Spanning trees}

Let ${\mathcal{G}}$ be a finite connected weighted graph. A \emph{spanning
tree} on $\mathcal{G}$ is a subgraph $(V,A)$, where $A$ is a set of edges,
which contains no cycles and is connected. We weight each spanning tree by
the product over edges in the tree of their weight and call the
corresponding probability measure the \emph{random spanning tree} on~${%
\mathcal{G}}$. 

By the transfer current theorem of Burton and Pemantle~\cite%
{BP}, the transfer current is a kernel for the random spanning tree measure,
that is, for any distinct edges $e_1,\ldots,e_k$, the probability that the random spanning tree contains them is 
\begin{equation*}
\P (e_1,\ldots,e_k)=\det \left(T(e_i,e_j)\right)_{1\le i,j\le k}\,.
\end{equation*}
This along with Theorem~\ref{ticv} implies Theorem~\ref{corrUST} below.

Given an edge $e$ in $\mathcal{G}$, we let $1_{e}$ denote the random variable
that takes value $1$ if $e$ belongs to the spanning tree, and $0$ otherwise (this
is an example of a pattern field, defined in Section~\ref{dpp} below). For $k$ distinct edges $e_1,\ldots,e_k$, the covariance of the corresponding random variables $1_{e_i}$ is $\cov\left(
e_{1},\ldots,e_{k}\right) =\mathbb{E}\left[ \left( 1_{e_{1}}-\mathbb{P}\left(
e_{1}\right) \right) \ldots\left( 1_{e_{k}}-\mathbb{P}\left( e_{k}\right)
\right) \right] $.

\begin{theorem}\label{corrUST}
Under the assumptions of Theorem~\ref{ticv}, the rescaled correlations of the spanning tree model
have a universal and conformally covariant limit.
\end{theorem}

\begin{proof}
Standard linear algebra calculation implies (see, e.g.
Lemma 21 of \cite{Ken00}) $$\cov\left( e_{1},...,e_{k}\right) =\det \left(
T(e_{i},e_{j})\delta _{i\neq j}\right) _{1\leq i,j\leq k}\,.$$ 
Applying Theorem %
\ref{ticv}, we obtain that, as $n\to\infty$,%
\begin{equation*}
n^{kd}\cov\left( e_{1},...,e_{k}\right) \rightarrow \det \left( c\left(
e_{j}\right) dg_{D}|_{\left( z_{i},z_{j}\right) }\left( e_{i},e_{j}\right)
\delta _{i\neq j}\right) _{1\leq i,j\leq k}.
\end{equation*}%
Therefore the limit of rescaled correlations exists, and its conformal
covariance follows from the conformal covariance of $dg_{D}$.
\end{proof}

\subsection{Two-component spanning forests and spanning unicycles}

\label{treescrst}

A \emph{two-component spanning forest} ($2$SF) on~${\mathcal{G}}$ is a
subgraph $(V,B)$, where $B$ is a set of edges, which contains no cycles and
has exactly two connected components. A \emph{spanning unicycle} (or \emph{%
cycle-rooted spanning tree}) is a connected subgraph $(V,B)$ where $|B|=|V|$
(it thus contains a unique cycle). See a picture on Figure~\ref{samples}. On
planar graphs these two notions are dual to one another. In the following we do not need to suppose however (unless otherwise stated) that the graphs are planar.

As for spanning trees, we consider the weight of a subgraph to be the
product over edges (of the subgraph) of the edge-weights. We define a probability measure on $2$SFs (respectively, random spanning
unicycles) by giving a $2$SF (respectively, a spanning unicycle) a probability proportional to its weight. This defines the random $2$SFs and random spanning unicycles considered hereupon.

A spanning unicycle can be thought of as a spanning tree to which an edge is
added. Note however, that the measure obtained from taking a random spanning
tree and adding a uniformly random edge is different than the random
spanning unicycle defined above, and the Radon-Nikodym derivative is simply
the length of the cycle. Nevertheless, we observe that the law of the cycle
of the random spanning unicycle conditional on the event that the cycle
contains some edge $e_1$ is the law of the path between the extremities of $%
e_1$ in the spanning tree model on the graph where $e_1$ is removed. The
event that the cycle is of a given shape is therefore the union of
translates of events for the spanning tree, which can be evaluated
explicitly for infinite periodic graphs. In particular for $\Z^2$, using the explicit values of the transfer current computed in \cite{BP}, we obtain that
the probability that the cycle is of length $4$ is $-\frac{16}{\pi ^3}+%
\frac{8}{\pi ^2}\approx .294$.

\begin{figure}[ht]
\centering
\begin{tabular}{ccc}
\includegraphics[width=7cm,height=7cm]{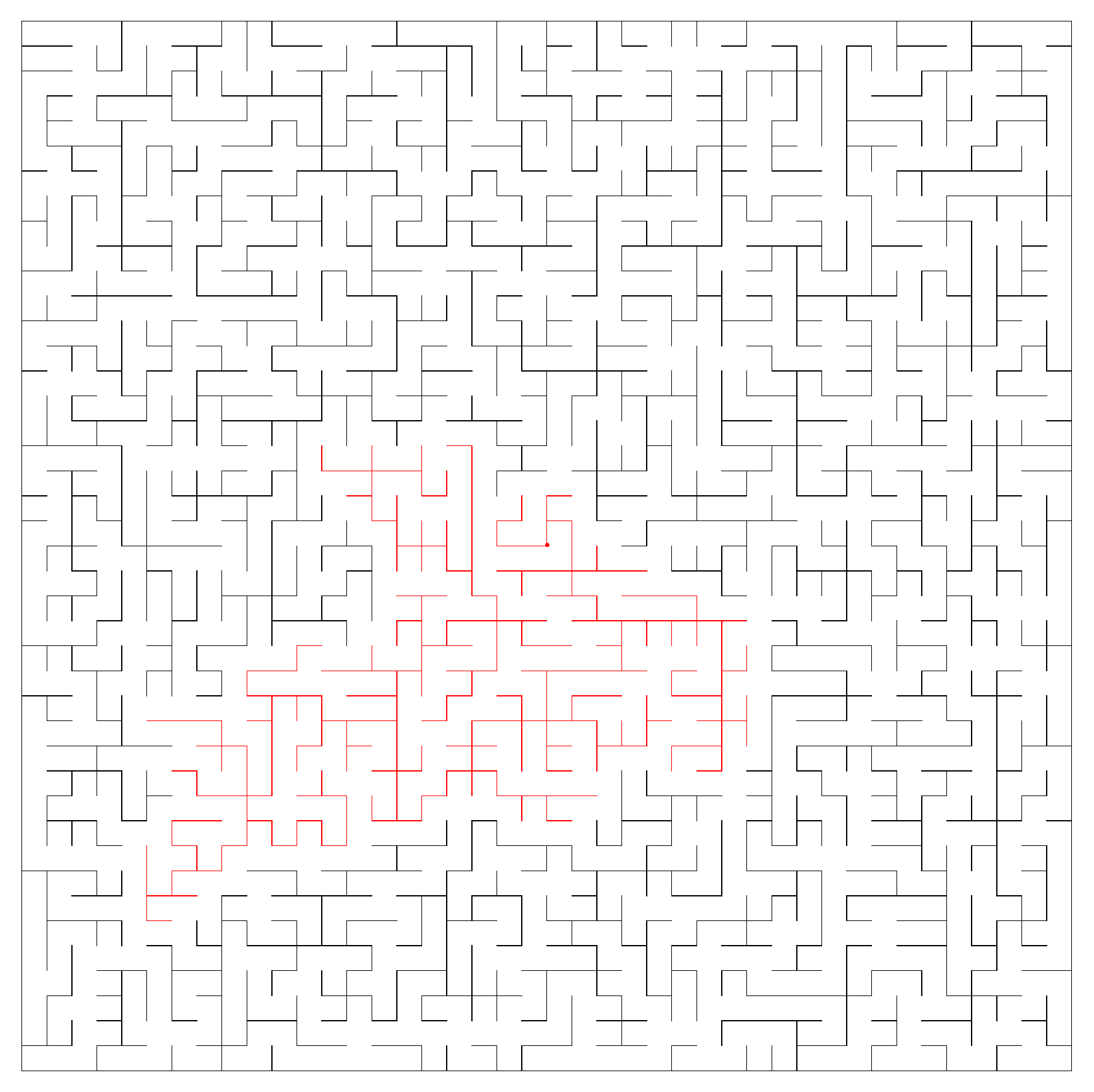} &  & %
\includegraphics[width=7cm,height=7cm]{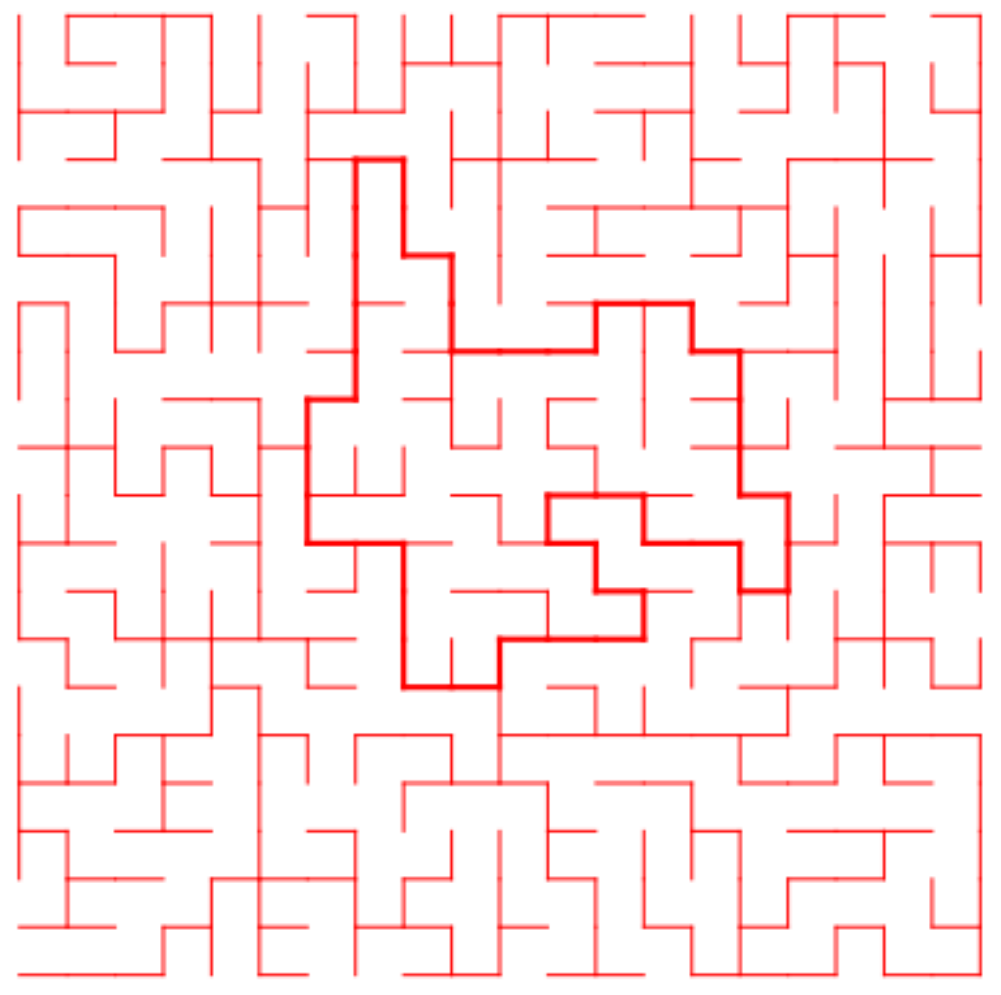}%
\end{tabular}%
\caption{Left: a two-component spanning forest wired on the boundary; Right:
a spanning unicycle on a $21\times 21$ square grid conditioned on having a
longish cycle}
\label{samples}
\end{figure}

Let $\kappa=\sum_{\mathcal{T}}\text{weight}(\mathcal{T})$ be the weighted sum of spanning trees
and $\lambda=\sum_{\mathcal{U}}\text{weight}(\mathcal{U})$ the weighted sum of spanning
unicycles. In the following, whenever $e$ denotes a directed edge $(u,v)$, we say that $u$ is its head, and $v$ its tail. Furthermore, the symbol $-e$ denotes the directed edge with same support and opposite direction, that is, $(v,u)$.

In the wired case ($\partial V\ne\emptyset$), using the fact (Theorem~3 of 
\cite{KKW}) that $\kappa G(x,y)$ is the weighted number of two-component
spanning forests such that $x$ and $y$ are disconnected from the boundary,
we have for any two directed edges $e$ and $e'$, 
\begin{equation}
T(e,e^{\prime })=\kappa ^{-1}\left( N_{e,e^{\prime }}-N_{e,-e^{\prime
}}\right) \,,  \label{trees}
\end{equation}%
where $N_{e,e^{\prime }}$ is the weighted number of two-component forests
such that the extremities of $e$ belong to different components, as well as
for the extremities of $e^{\prime }$, and these are the same for both heads
and both tails.

In the free boundary case ($\partial V=\emptyset$), we start with the
following result.

\begin{lemma}[Lemma 1.2.12 in~\cite{K13}]
\label{projector} Let $U$ be the set of all spanning unicycles. For any $%
\theta\in\Omega^1$, we have 
\begin{equation*}
\sum_{\mathcal{U}\in U}\theta_{\gamma}^2=\kappa\left\langle\theta,(1-T)\theta\right\rangle\,,
\end{equation*}
where $\theta_\gamma=\langle\theta,\delta_\gamma\rangle$ and $\gamma$ is the
choice of a directed representative of the unique cycle of the unicycle~$\mathcal{U}$.
\end{lemma}

\begin{proof}
It is an immediate corollary of Proposition 7.3 of~\cite{Bi}. The map $T$ is the orthogonal projection on $\mathrm{Im}(d)$, which is the orthogonal complement of $\mathrm{Ker}(d^*)$. Recall from~\cite{Bi} that a fundamental cycle associated to a spanning tree and a directed edge not on the tree, is the directed cycle defined by the union of that edge with the unique path joining the extremities of that edge in the tree (the direction of the cycle is given by the direction of the edge). A basis of $\mathrm{Ker}(d^*)$ is given by the ($1$-forms corresponding to) fundamental cycles associated to a spanning tree. Hence
$$T=1-\kappa^{-1}\sum_{\mathcal{T}\,\text{spanning tree}}\sum_{e\notin \mathcal{T}}1_{C(e,\mathcal{T})}\,,$$ where $C(e,\mathcal{T})$ is the fundamental cycle of $e$ in $\mathcal{T}$.
\end{proof}

\begin{remark}
Lemma~\ref{projector} may also be proved by expanding the determinant of the line bundle Laplacian (introduced in~\cite{Kenyon}) around the trivial connection along $\theta$ (that is, expand near $t=0$ the family of Laplacian determinants associated to the connection $\exp(t i\theta)$ ). This yields yet another geometric interpretation of the transfer current.
\end{remark}

In the planar case, we obtain the following. Taking $\theta $ to be the
indicator on directed edge $e^{\ast },e^{\prime \ast }$ on the dual graph,
we have 
\begin{equation}
T^{\ast }(e^{\ast },e^{\prime \ast })=\kappa ^{-1}\left( N_{e^{\ast
},e^{\prime \ast }}-N_{e^{\ast },-e^{\prime \ast }}\right) \,,  \label{crst}
\end{equation}%
where $N_{e^{\ast },e^{\prime \ast }}$ is the number of spanning unicycles,
the cycle of which contains $e^{\ast }$ and~$e^{\prime \ast}$. In the planar
case, the dual of a $2$SF is a spanning unicycle. It follows from~%
\eqref{trees} and~\eqref{crst} that $T(e,e^{\prime })=T^{\ast }(e^{\ast
},e^{\prime \ast })$, which gives another proof of Corollary~\ref{tidual}
above.

Lemma~\ref{projector} gives a way to compute the expected winding of the
cycle of the spanning unicycle.

We conclude this subsection by giving two applications of Theorem~\ref%
{projector} to the case of planar graphs or graphs embedded in ${\mathbb{R}}%
^3$.

For planar graphs with two marked faces $f$ and $f$', let $Z$ and $Z^{\prime
}$ be the set of east-directed edges crossed by two disjoint dual paths from
the outer face to $f$ and~$f$', respectively. Taking $\theta $ to be the $1$%
-form indicator of $Z$ and $Z^{\prime }$, we obtain that the probability
that two faces lies inside the cycle of the uniform spanning unicycle on a
planar graph is 
\begin{equation*}
\P (f,f^{\prime })=-\frac{\kappa }{\lambda }\sum_{e\in Z,e^{\prime }\in
Z^{\prime }}T(e,e^{\prime })\,.
\end{equation*}%
This formula simplifies to $\P (f,f^{\prime })=\left( \kappa /\lambda
\right) G^{\ast }(f,f^{\prime })$ by Corollary~\ref{tidual} and was observed
in~\cite{KKW}. For any isoradial graphs, the ratio $|V_{n}|\kappa /\lambda $
converges to a constant $\tau $, see~\cite{KW13} (and~\cite{LP} for the case
of ${\mathbb{Z}}^{2}$).

In the case of ${\mathbb{Z}}^3$, we obtain the following.

\begin{corollary}
Consider ${\mathcal{G}}_n$ to be the subgraph of ${\mathbb{Z}}^3$ whose
vertex-set is $\llbracket 1,n\rrbracket^3$. Let $f$ be a face in the $x,y$ plane with coordinates $(x,y)$. Let $\k$ be
the winding number of the uniform spanning unicycle in~${\mathcal{G}}_n$ around the tube $f%
\times\llbracket 1,n\rrbracket$. Its second moment is given by 
$$
\mathbb{E}(\k^2)=\frac{1}{\tau_n n^3}\left(ny-\sum_{y_i,y_j=1}^{y}%
\sum_{z,z^{\prime }=1}^{n}T(e_i,e^{\prime }_j)\right)\,, 
$$
where $\tau_n\to\tau$, and $\tau$ is a constant.
\end{corollary}

\begin{proof}
Consider a $1$-form $\theta$ defined to be the indicator of the edges on a directed ``curtain" of edges that connect the tube $f\times \llbracket 1,n\rrbracket$ to the boundary of $\llbracket 1,n\rrbracket^3$: more precisely we consider a path $\pi$ in the $xy$ plane from $f$ to the boundary and the curtain consists of the edges on $\pi\times \llbracket 1,n\rrbracket$ oriented consistently. In that way, any oriented simple closed curve~$\gamma$ satisfies $\left\langle \theta,1_\gamma\right\rangle=k$ where $k$ is the algebraic winding number of that curve around the tube. 

By Lemma~\ref{projector} we therefore have
\begin{align*}
\E(\k^2)&=\frac{\kappa}{\lambda}\left\langle\theta,(1-T)\theta\right\rangle\\
&=\frac{\kappa}{\lambda}\bigl(ny-\sum_{y_i,y_j=1}^{y}\sum_{z,z'=1}^{n}T(e_i,e'_j)\bigr)\,, 
\end{align*}
where $\tau_n=\lambda/(\kappa n^3)$ converges to a constant as shown in~\cite{LP}.
\end{proof}

Higher dimensional analogs of the previous results can be derived using
Lemma~\ref{projector} and the fact that the ratio $\tau_n$ converges by~\cite%
{LP}.

\subsection{Minimal subconfigurations of the abelian sandpile}

The \emph{abelian sandpile model} is constructed as the stationary
distribution of a certain Markov chain on integer functions over interior
vertices of a graph. A sandpile configuration is a particle configuration $%
\eta :V^{\circ}\rightarrow \mathbb{N}$. If for some $x\in V^{\circ}$ we have 
$\eta \left( x\right) \geq \deg x$, then the vertex $x$ may topple, that is
send one particle to each of its neighbors (this corresponds to the
operation $\eta \mapsto \eta -\Delta (\delta _{x})$). Particles that reach
the boundary $\partial V$ (referred to as the \emph{sink} in this context)
are lost. A sandpile is stable if $\eta \left( x\right) <\deg x$ for any $%
x\in V^{\circ}\setminus \{s\}$. A bilinear operation $\oplus$ can be defined
on the set of stable configurations: addition in $\mathbb{Z}^{V}$ followed
by toppling until stabilization. (The order in which toppling occurs does
not matter, it is an abelian network~\cite{D}.)

The Markov chain on the set of stable sandpiles is the following. At any
given time, add (using $\oplus $) a particle to the configuration at a
uniformly chosen vertex in~$V^{\circ }$. It was shown in \cite{D} that the
stationary distribution of this chain is unique and is uniform on the set of
recurrent states. The \emph{recurrent configurations} of the sandpile are defined as
the recurrent states of this Markov chain.

The set of recurrent configurations has a local description given by the
burning algorithm~\cite{D}. A stable configuration is recurrent if and only
if any subconfiguration satisfies that it is unstable regarded as a
configuration on the subgraph on which it lives. Using this criterion,
Majumdar and Dhar constructed a bijection (called the burning bijection)
mapping recurrent configurations of the abelian sandpile to spanning trees~%
\cite{MD,MD2}. It depends on the choice of an ordering of edges around each
vertex. See~\cite{Ruelle} for a theoretical physics perspective on the
sandpile as a field theory.

In the case of a weighted graph, we define a measure $\nu$ on the set of
recurrent configurations which is the pullback of the weighted spanning tree
measure under the burning bijection. There is a way to relate $\nu$ to the
discrete projection of the continuous sandpile (whose dynamics is analog to
the one described above for unweighted graphs~\cite{Gabr,KW13}).

Under the burning bijection, any local event for the sandpile therefore
translates into an event for the spanning tree, however not necessarily
local. An important concept is that of \emph{minimal subconfiguration} (the
original definition is from~\cite{DM90} where it is called weakly allowed
subconfiguration).

A subconfiguration on a subgraph $W$ is \emph{minimal} if it is part of a
recurrent configuration, but by decreasing any of its heights, this is no
longer true. In particular, conditional on a minimal subconfiguration, the
measure on the outside of $W$ is the sandpile measure on $G\setminus W$,
still with sink at $\partial V$.

The easiest example of a minimal subconfiguration is a single vertex with
height $0$, or any collection of vertices with $0$ height provided none of
these vertices are neighbors. There are minimal subconfigurations on any
subgraph $W$ of ${\mathcal{G}}$, see Figure~\ref{MinimalConfig} for a more
elaborate example.

\begin{figure}[ht]
\centering
\includegraphics[width=6cm,height=4cm]{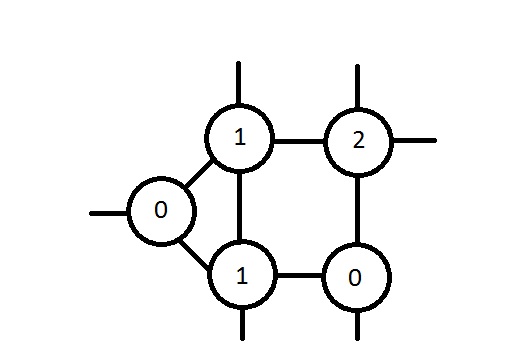}
\caption{Example of a minimal subconfiguration}
\label{MinimalConfig}
\end{figure}

Let ${\mathcal{G}}_{W}$ be the graph obtained by wiring all the vertices in $%
G\backslash W$ with $\partial V$, and removing all self edges. The following
proposition is the generalization of Theorem~1 of~\cite{JW} to the weighted
graph setting.

\begin{proposition}
\label{minimaldet}
Given any spanning tree $\mathcal{T}$ of $\G_{W}$, let $\mathcal{E}$ be the edge set  
\begin{equation*}
\mathcal{E}\text{=}\left\{ \left( x,y\right) :x\in W\text{, }\left\{
x,y\right\} \notin \mathcal{T}\right\}\,.
\end{equation*}
For any minimal configuration $\xi$ supported on $W$, there is a spanning tree $\mathcal{T_0}$ of $\G_W$ (given by our choice of burning bijection) such that
\begin{equation*}
\P (\eta _{W}=\xi )=\det (I-T)_{\mathcal{E_0}}=\sum_{\mathcal{T}}\frac{\mathrm{weight}\left( \mathcal{T_0}\right) }{\sum_{\mathcal{T}}\mathrm{weight}\left( \mathcal{T}\right)}
\det (I-T)_{\mathcal{E}}\,,
\end{equation*}
where the sums are over spanning trees of~$\G_W$.
\end{proposition}

\begin{proof}
By Lemma 4 in~\cite{JW}, the minimal subconfigurations are mapped (by the burning bijection) to spanning trees for which the projection of the tree on $\G_W$ is again a tree.
Therefore, the probability of a minimal configuration is the probability that the small tree is the image $\T_W$ of the subconfiguration on $\G_W$. By the Gibbs property of the weighted spanning tree measure, this is simply the probability that the tree measure on $\G$ does not contain the edges of~$\mathcal{E}_W$.
Under the choice of orientation, the probability that (for a fixed subconfiguration) we map it to $\T_W$ is proportional to the edge weights in the tree.
\end{proof}

Proposition~\ref{minimaldet} and Theorem~\ref{ticv} imply the following.

\begin{theorem}\label{corrASM}
Under the hypotheses of Theorem~\ref{ticv}, the rescaled correlations of a finite number of macroscopically distant minimal subconfigurations has a universal scaling
limit. In particular, the rescaled correlations of zero heights
converge to a universal scaling limit (which is conformally covariant in
dimension $2$).
\end{theorem}

\subsection{Discrete Gaussian free field current flow}

The \emph{discrete Gaussian free field} (DGFF) $\Gamma$ (with free or wired
boundary conditions) is the Gaussian vector with zero mean and covariance
matrix given by Green's function (with free or wired boundary conditions).

The current flow of the DGFF, defined by $J=d\Gamma$, is a Gaussian $1$-form
with zero mean and covariance matrix given by the transfer current, since
for any $1$-form $\alpha$, we have 
\begin{align*}
\mathbb{E}((J\alpha)^2)&=\mathbb{E}((\Gamma d^*\alpha)^2)=\left\langle
d^*\alpha,G d^*\alpha \right\rangle \\
&=\left\langle\alpha,dGd^*\alpha\right\rangle=\left\langle\alpha,T\alpha%
\right\rangle\,,
\end{align*}
where we used the fact that $d$ and $d^*$ are dual.

Whenever the discrete Green function converges to the continuous one, the
discrete Gaussian free field converges weakly in distribution to the
continuous Gaussian free field. In virtue of Theorem~\ref{ticv}, it is also
the case for the current flow, which converges to its continuous counterpart.

\begin{theorem}\label{dgff}
Under the hypotheses of Theorem~\ref{ticv}, the flows $J_n$ converge weakly in distribution to $d\Gamma$, where $\Gamma$ is the GFF and the derivative is in the distributional sense.
\end{theorem}

\section{Random point fields}

\label{dpp}

\subsection{Patterns}

A simple point process (point process for short) on a discrete~\footnote{%
The result of this section easily adapts to the case of (uncountable) Polish
spaces but we restrict to the discrete case to avoid cumbersome
technicalities.} space $\Omega$ is a random subset ${\mathcal{T}}%
\subset\Omega$, or equivalently, a random configuration of $\{0,1\}$-valued
spins $\sigma_x$ at each point $x\in\Omega$. We define a \emph{pattern} to
be a finite (deterministic) spin configuration which can be represented as a
pair of disjoint sets $\mathbf{x}=(\mathbf{x}_0,\mathbf{x}_1)$, where $%
\mathbf{x}_0$ is the set of points with $0$ spin and $\mathbf{x}_1$ the set
of points with spin $1$. The \emph{support} of a pattern is the set $\mathbf{%
x}_0\cup\mathbf{x}_1$ which we denote by $\{\mathbf{x}\}$. We say that two
patterns are disjoint when their supports are disjoint. A mono-pattern is a
pattern where all spins have the same value.

Given a collection of patterns $X$, a point process $\xi$ on $\Omega$
induces a point process~$\xi_X$ on the disjoint union of the patterns, which
we may still view as a point process on~$\Omega$ (for example, by fixing an
ordering of points in $\Omega$ and identifying the pattern with its element
of least index). This is the \emph{pattern field} associated to $X$.

Formally, a pattern field can be represented as $\xi _{X}=\sum_{\mathbf{x}%
\in X}\delta _{\mathbf{x}}\,$, which is an element of the dual of $\ell
^{2}(\Omega )$. We will omit the dependence on $X$ if it is clear from the
context. For any $f\in \ell ^{2}(\Omega )$, we write $\xi (f)=\langle \xi
,f\rangle $ and $\xi (A)$ if $f=1_{A}$ is the indicator of a subset~$%
A\subset \Omega $.

\subsection{Symmetric and block determinantal processes}

Let $K$ be the kernel of a trace-class self-dual map in $\ell ^{2}(\Omega )$
with eigenvalues in $[0,1]$, that is a symmetric matrix indexed by $\Omega $%
. For any finite set $A$, denote $K_{A}$ the restriction of~$K$ onto~$\ell
^{2}(A)$. Let $\widehat{I}_A$ be the diagonal matrix with $1$s on the entries indexed by $A$ and $0$ elsewhere. 
$K$ defines a symmetric determinantal process, which means that for any
pattern $\mathbf{x}$, we have

\begin{equation}
\P (\mathbf{x})=\det\left(-\widehat{I}_{\mathbf{x}_0}+K\right)_{\{\mathbf{x}\}}\,.
\end{equation}
The kernel $I-K$ is a kernel for the complement process $\Omega\setminus{%
\mathcal{T}}$.

Although the underlying point process is determinantal, the pattern field is
not. However, joint probabilities of patterns can still be written as a
large minor of a modified matrix. More precisely,%
\begin{equation}
\P (\mathbf{x}^{\left( 1\right) },...,\mathbf{x}^{\left( k\right) })=\det
\left( -\widehat{I}_{\cup _{j=1}^{k}\mathbf{x}_{0}^{\left( j\right) }}+K\right) _{\cup
_{j=1}^{k}\{\mathbf{x}^{(j)}\}}.  \label{kpattern}
\end{equation}
Therefore it can be viewed as a block determinantal process, see Section~%
\ref{blockdet} for more details.

It follows from a short computation that for any finite set $A\subset \Omega$%
, we have 
\begin{equation}  \label{detBernoulli}
\mathbb{E}\left(z^{\xi(A)}\right)=\det (zK_{A}+I-K_{A})=\prod_{\lambda \in 
\mathrm{Spec}(K_{A})}(z\lambda +(1-\lambda ))\,,
\end{equation}
which signifies that the number of points in $A$ is a sum of independent
Bernoullis. These facts are well-known, see e.g. Chapter 4 of~\cite{HKPV}
for an introduction to determinantal processes.

\begin{proposition}
\label{patternBernoulli} Let $\xi $ be a symmetric determinantal point
field. Let $X$ be a family of disjoint patterns. The number of points
observed in $A\cap (\cup _{\mathbf{x}\in X}\{\mathbf{x}\})$, where $A\subset
\Omega $ is any finite subset, is a sum of independent Bernoullis.
\end{proposition}

\begin{proof}
Since the patterns are disjoint, for any finite set $A\subset \Omega $, we
have $\xi (A\cap (\cup _{\mathbf{x}\in X}\{\mathbf{x}\}))=\xi (A\cap (\cup _{%
\mathbf{x}}(\mathbf{x}_{1}^{c}\cup \mathbf{x}_{0})))$. The result follows by applying~\eqref{kpattern} to obtain~\eqref{detBernoulli}.
\end{proof}

For any finite set of patterns $A$, we have 
\begin{equation*}
\mathrm{Var}(\xi _{X}(A))=\sum_{\mathbf{x}\in A}\sum_{\mathbf{y}\in A}%
\mathrm{Cov}(\mathbf{x,y})\,.
\end{equation*}%
In the case when the patterns are singletons, $A$ is just a subset of $%
\Omega $, and 
\begin{align*}
\mathrm{Cov}(x,y)& =(K(x,x)(1-K(x,x))1_{\{x=y\}}-K(x,y)^{2}1_{\{x\neq y\}} \\
& =(\P (x)-K(x,x)^{2})1_{\{x=y\}}-K(x,y)^{2}1_{\{x\neq y\}}\,.
\end{align*}%
Thus the variance is given by 
\begin{equation*}
\mathrm{Var}(\xi (A))=\sum_{x\in A}\left( K\left( x,x\right) -\sum_{y\in
A}K\left( x,y\right) ^{2}\right) .
\end{equation*}

\subsection{{Fluctuations for block determinantal processes}\label{blockdet}}

Before studying the scaling limits of pattern fields, we show (following
ideas of Boutillier~\cite{Bo}) a fluctuation result for general block
determinantal fields (without assuming the kernels to be symmetric). In
combination with the covariance structure studied in the next section, this
implies that under the right decay of correlation, pattern fields of
determinantal processes converge to Gaussian white noise.

Let $d\geq 1$ and $D\subset \mathbb{R}^{d}$ be a domain. Let $\Omega $ be
the space of locally finite particle configurations in $D$. Denote $\mathcal{%
F}$ to be the $\sigma -$measurable subsets of $\Omega $, generated by the
cylinder sets $$\mathcal{C}_{B}^{n}=\left\{ \xi \in \Omega :\left\vert \text{%
particles of }\xi \text{ that lies in }B\right\vert =n\right\}\,,$$ where $%
B\subset D$ is a Borel set and $n\in \mathbb{N}$. $\left( \Omega ,\mathcal{F}%
,\mathbb{P}\right) $ is said to be a block determinantal process, if for any 
$k\in \mathbb{N}$ and $\left\{ x_{i}\right\} _{i=1}^{k}\subset D$,%
\begin{equation*}
\mathbb{P}\left( \#\left( dx_{i}\right) =1,i=1,...,k\right) =\rho _{k}\left(
x_{1},...,x_{k}\right) dx_{1}...dx_{k},
\end{equation*}%
where the $k$ point correlation function 
\begin{equation}
\rho _{k}\left( x_{1},...,x_{k}\right) =\det \left( A_{x_{i}x_{j}}\right)
_{i,j=1}^{k}\text{,}  \label{kcorr}
\end{equation}%
and each $A_{x_{i}x_{j}}$ is a $m_{i}\times m_{j}$ matrix. We will abuse the
notation below and let $\mathbb{P}\left( x_{1},...,x_{k}\right) $ denote $%
\rho _{k}$.

\begin{proposition}
\label{blockclt} Let $\left( \Omega ,\mathcal{F},\mathbb{P}^{n}\right) _{n>0}
$ be a family of block determinantal processes. Assume that $\det
A_{x_{i}x_{i}}^{n}$ is bounded from below by a strictly positive constant, and $%
A_{x_{i}x_{j}}^{n}$ converges as $n\rightarrow \infty $ to a limit $%
A_{x_{i}x_{j}}$ satisfying $\left\Vert A_{x_{i}x_{j}}\right\Vert \leq
O\left( |x_{i}-x_{j}|^{-d}\right) $ as $\left\vert x_{i}-x_{j}\right\vert
\rightarrow \infty $. Then a rescaling-recentering of the associated point
field $\xi _{n}$ converges weakly in distribution to a Gaussian field: for
any test function $\varphi \in C_{0}^{\infty }(D)$, we have 
\begin{equation*}
\frac{\xi _{n}(\varphi )-\E(\xi _{n}(\varphi ))}{\var(\xi _{n}(\varphi
))^{1/2}}\rightarrow \mathcal{N}\left( 0,1\right) \,.
\end{equation*}
\end{proposition}

We will prove the proposition by verifying Wick's theorem for the counting
field, and shows the finite dimensional distributions converge to that of a
Gaussian field (see e.g.~\cite{Janson}). The Wick's theorem states if $\xi $
is a random field such that all of its moments can be expressed in terms of
the covariance structure as follows: for any smooth test functions $\varphi
_{1},...,\varphi _{n}$,%
\begin{equation}
\mathbb{E}\left( \xi \left( \varphi _{1}\right) ...\xi \left( \varphi
_{n}\right) \right) =\left\{ 
\begin{array}{cc}
\sum_{\text{pairings}}\prod\limits_{i=1}^{n/2}\mathbb{E}\left( \xi \left(
\varphi _{i}\right) \xi \left( \varphi _{\sigma \left( i\right) }\right)
\right)  & n\text{ even,} \\ 
0 & n\text{ odd,}%
\end{array}%
\right.   \label{wick}
\end{equation}%
then $\xi $ is Gaussian. In practice, by polarization identities it suffices
to check the above identity for $\mathbb{E}\left( \xi \left( \varphi \right)
^{n}\right) $.

The basic idea is that when computing higher moments, all the contributions
not coming from pair correlations vanish, since the correlations decay fast
enough. The argument follows the spirit of~\cite{Bo}.

\begin{proof}
We first verify the Wick formula in the case where the contribution points
are macroscopically distant. Let 
\begin{equation*}
\Xi _{k}^{n}\left( \varphi \right) =\mathrm{Var}\left( \xi _{n}\left(
\varphi \right) \right) ^{-k/2}\sum_{\substack{ z_{1},...,z_{k} \\ \text{%
distinct}}}\varphi \left( z_{1}\right) ...\varphi \left( z_{k}\right) 
\mathbb{E}_{n}\left( \prod\limits_{i=1}^{k}\left( 1_{z_{i}}-\mathbb{P}%
^{n}\left( z_{i}\right) \right) \right) .
\end{equation*}

Note the following exclusion-inclusion formula 
\begin{align*}
&\mathbb{E}_{n}\left[ \left( 1_{z_{1}}-\mathbb{P}^{n}\left( z_{1}\right)
\right) ...\left( 1_{z_{k}}-\mathbb{P}^{n}\left( z_{k}\right) \right) \right]
\\
&=\sum_{i_{1},...,i_{p}}\prod\limits_{i\notin \left\{
i_{1},...,i_{p}\right\} }\left( -\mathbb{P}^{n}\left( z_{i}\right) \right) 
\mathbb{E}_{n}\left[ \prod\limits_{l=1}^{p}1_{z_{i_{l}}}\right]  \\
&=\prod\limits_{i=1}^{k}\mathbb{P}^{n}\left( z_{i}\right) \sum_{C\subset
\left\{ 1,...,k\right\} }\left( -1\right) ^{k-\left\vert C\right\vert }\det
\left( 1_{\left\{ i,j\right\} \in C}B_{z_{i}z_{j}}^{n}\right) _{i,j=1}^{k}\,,
\end{align*}%
where $B_{z_{i}z_{j}}^{n}=\left( \mathbb{P}^{n}\left( z_{i}\right) \right)
^{-1}A_{z_{i}z_{j}}^{n}$.

Now expand the determinant, and notice that the off-diagonal entries of the
matrix $\left\{ 1_{\left\{ i,j\right\} \in C}B_{z_{i}z_{j}}^{n}\right\} $
are small (bounded by $O\left( n^{-d}\right) $). One can rewrite the above
expression in terms of sums over permutations fixing no block $%
B_{z_{i}z_{i}}^{n}$ (or equivalently, as products of cycles), and the
support of which intersect each block exactly once. We thus obtain%
\begin{align*}
&\mathbb{E}_{n}\left[ \left( 1_{z_{1}}-\mathbb{P}^{n}\left( z_{1}\right)
\right) ...\left( 1_{z_{k}}-\mathbb{P}^{n}\left( z_{k}\right) \right) \right]
\\
&=\sum_{C\subset \left\{ 1,...,k\right\} }\left( -1\right) ^{k-\left\vert
C\right\vert }\prod\limits_{i=1}^{k}\mathbb{P}^{n}\left( z_{i}\right)
\sum_{S \text{ fixing no block}}\text{sgn}\left( S \right)
\prod\limits_{i=1}^{k}\prod\limits_{\alpha =1}^{m_{i}}B_{\left( z_{i},\alpha
\right) ,S \left( z_{i},\alpha \right) }^{n} \\
&=\sum_{C\subset \left\{ 1,...,k\right\} }\left( -1\right) ^{k-\left\vert
C\right\vert }\prod\limits_{i=1}^{k}\mathbb{P}^{n}\left( z_{i}\right) \sum
_{\substack{ S\in S_{k} \\ \text{fix no point}}}\text{sgn}\left( S\right)
\left( \prod\limits_{i=1}^{k}\sum_{\alpha _{i}=1}^{m_{i}}B_{\left(
z_{i},\alpha _{i}\right) ,\left( z_{S\left( i\right) },\alpha _{S\left(
i\right) }\right) }^{n}\right)  \\
&+O\left( n^{-2d}\right) \,,
\end{align*}%
where the $O\left( n^{-2d}\right) $ term accounts for the contribution from
permutations that has at least two non-fixed point in some block. Indeed,
for a particular $S$ that partitions $\left\{ 1,...,k\right\} $ into cycles $%
\left\{ \gamma _{l}\right\} $, one can write%
\begin{equation*}
\prod\limits_{i=1}^{k}\sum_{\alpha _{i}=1}^{m_{i}}B_{\left( z_{i},\alpha
_{i}\right) ,\left( z_{S\left( i\right) },\alpha _{S\left( i\right) }\right)
}^{n}=\prod\limits_{\left\vert \gamma_{l} \right\vert =p}\text{Tr}\left(
B_{z_{1}z_{2}}^{n}...B_{z_{p}z_{1}}^{n}\right) .
\end{equation*}

Let $\left\{ \Gamma _{l}\right\} $ denote some partition of $\left\{
1,...,k\right\} $. Therefore, 
\begin{align*}
\Xi _{n}^{\varepsilon }\left( \varphi \right)&=\sum_{C\subset \left\{
1,...,k\right\} }\left( -1\right) ^{k-\left\vert C\right\vert
}\prod\limits_{i=1}^{k}\mathbb{P}^{n}\left( z_{i}\right) \sum_{\left\{
\Gamma _{l}\right\} }\prod\limits_{l}\sum_{\substack{ \gamma \text{: }%
\left\vert \gamma \right\vert =m \\ \text{supp}\gamma =\Gamma _{l}}}\text{sgn%
}\left( \gamma \right) \left( \mathrm{Var}\left( \xi _{n}\left( \varphi
\right) \right) \right) ^{-m/2} \\
&\cdot \left( \sum_{\substack{ z_{1},...,z_{m} \\ \text{distinct}}}\varphi
\left( z_{1}\right) ...\varphi \left( z_{m}\right) \text{Tr}\left(
B_{z_{1}z_{2}}^{n}...B_{z_{m}z_{1}}^{n}\right) +O\left( n^{-2d}\right)
\right)\,.
\end{align*}%
It suffices to show all the contributions from a cycle of length $m\geq 3$
vanish. It is easy to verify that%
\begin{equation*}
\var\left( \xi _{n}(\varphi )\right) =\sum_{z,z^{^{\prime }}}\varphi \left(
z\right) \varphi \left( z^{\prime }\right) \P _{n}\left( z,z^{\prime
}\right) -\left( \sum_{z}\varphi \left( z\right) \P _{n}\left( z\right)
\right) ^{2}=O\left( n^{2d}\right) \,.
\end{equation*}%
Therefore each term arising from a cycle of length $m\geq 3$ is bounded by 
\begin{align}
&Mn^{-m}\sum_{\substack{ z_{1},...,z_{m} \\ \text{distinct}}}\text{Tr}%
\left( B_{z_{1}z_{2}}^{n}...B_{z_{m}z_{1}}^{n}\right)   \notag \\
&\leq MC_{m}n^{-m}\sum_{\substack{ z_{1},...,z_{m} \\ \text{distinct%
}}}\left\Vert B_{z_{1}z_{2}}^{n}\right\Vert ...\left\Vert
B_{z_{m}z_{1}}^{n}\right\Vert ,  \label{cycle}
\end{align}%
for some $M<\infty $. Since $\mathbb{P}_{n}\left( z_{i}\right) =O\left(
1\right) $, the decay rate of $B_{z_{i}z_{j}}$ is the same as~$A_{z_{i}z_{j}}
$, which implies that $\sum_{z_{2}\neq z_{1}}\left\Vert
B_{z_{1}z_{2}}^{n }\right\Vert ^{s}<\infty $ for any $s>1$. Apply
the following elementary inequality to (\ref{cycle}),%
\begin{equation*}
\prod\limits_{i=1}^{m}a_{i}\leq \frac{1}{m}\left[ \left(
a_{1}...a_{m-1}\right) ^{\frac{m}{m-1}}+...+\left(
a_{m}a_{1}...a_{m-2}\right) ^{\frac{m}{m-1}}\right] ,
\end{equation*}%
and note that each term can be bounded by $O\left( n^{2-m}\right) $. For
instance, the first corresponding term satisfies%
\begin{equation*}
C\sum_{\substack{ z_{1},...,z_{m} \\ \text{distinct}}}\left( \left\Vert
B_{z_{1}z_{2}}^{n}\right\Vert ...\left\Vert B_{z_{m}z_{1}}^{n}\right\Vert
\right) ^{\frac{m}{m-1}}<\infty \,,
\end{equation*}%
by successively summing over $z_{m},z_{m-1},...,z_{2}$. Therefore (\ref{cycle})
can be bounded by $O\left( n^{2-m}\right) $, which vanishes in the limit for $m\geq 3$.

All the terms contributing to the limit are thus coming from pairings, and
we verified Wick's formula~\eqref{wick}.

When the $z_{i}$'s are not necessarily all distinct, we can verify the Wick
formula for higher moments by the same argument as in Section 3.2.2 in~\cite%
{Bo}, and we omit the computation here.
\end{proof}

\subsection{Central limit theorem for pattern fields}

Suppose that $\Omega$ is infinite and embedded in ${\mathbb{R}}^d$ in a
discrete way (no accumulation points). Let $D\subset{\mathbb{R}}^d$ be a
simply connected domain containing~$0$. Define $\Omega_{n}=D\cap \Omega/n$:
these sets (when remultiplied by $n$) form an exhausting sequence of subsets
of $\Omega$. We suppose that each $\Omega_n$ is equipped with a
determinantal process with kernel $K_{n}$. We assume that $K_{n}\rightarrow K
$ pointwise. This implies that the point process on $\Omega$ is the weak
limit of the process defined on $\Omega _{n}$. As we will see, this implies
that scaling limits of fields constructed from local events do not feel the
influence of the boundary shape of $D$.

For each $n$, let $\mathcal{P}_n$ be a collection of patterns in $\Omega_n$.
When the choice of~$\mathcal{P}_n$ is fixed, we write~$\xi_n$ instead of~$%
\xi_{\mathcal{P}_n}$.

\begin{theorem}
\label{cltdet} 
Assume that $K(x,y)=O(|x-y|^{-d})$. If for any $n$ and $\mathbf{x}\in \mathcal{P}_{n}$, $\mathbb{P}_{n}\left( 
\mathbf{x}\right) $ is uniformly bounded from below, then as $\left\vert
A_{n}\right\vert \rightarrow \infty $, $(\xi_{n}(A_{n})-\mathbb{E}(\xi _{n}(A_{n})))/\mathrm{%
Var}(\xi_{n}(A_{n}))^{1/2}$ converges in law to a standard Gaussian.
\end{theorem}

In the case where the patterns are singletons, the only condition we require
is that $\mathrm{Var}(\xi_{n}(A_{n}))\to\infty$, and the result follows from
the Lindenberg--Feller theorem as explained in Theorem 4.6.1 of~\cite{HKPV}.

\begin{proof}
By pointwise convergence of $K_n$ to $K$, we have $K_{n}\left( x_{i},x_{j}\right)=O(n^{-d}\left\vert x_{i}-x_{j}\right\vert ^{-d})$ for any $x_{i},x_{j}\in \Omega _{n}$.
The result now follows from Proposition~\ref{blockclt} since the pattern field is block determinantal.
~\end{proof}

It follows from Theorem~\ref{cltdet} that the rescaled-recentered field
constructed from $\xi_n$ converges to a Gaussian field, that can be viewed
as a random element of $L^{2}(D)$. By the lower bound on $\P (x)$, the
growth rate of $\mathrm{Var}(\xi_n)$ is $|\mathcal{P}_n|$ (which we will
take to be of the order of $n^{d}$) and compute the covariance structure of
the Gaussian random field obtained as the limit of $n^{-d/2}(\xi_n-\mathbb{E}%
\xi_n)$.

To characterize the law of a Gaussian field, it suffices to compute $\mathbb{%
E}(\xi (\varphi )^{2})$ for any test function $\varphi $ (see~\cite{Janson}%
). If 
\begin{equation*}
\mathbb{E}(\xi (\varphi )^{2})=\iint \varphi \left( x\right) C\left(
x,y\right) \varphi \left( y\right) dx\,dy\,,
\end{equation*}%
then $\xi $ is a Gaussian field with correlation kernel $C$. When $C\left(
x,y\right) =I \delta(x-y)$, where $\delta$ is the Dirac mass and $I>0$, the
field is \emph{Gaussian white noise}, $\mathbb{E}(\xi (\varphi )^{2})$ is
proportional to~$|\varphi|_{L^{2}}^{2}$ and $I$ is called the \emph{intensity%
}.

Let ${\mathcal{P}}_{n}$ be an increasing sequence of finite collections of
disjoint patterns on $\Omega _{n}$, such that $\left\vert {\mathcal{P}}%
_{n}\right\vert =O\left( n^{d}\right) $. Denote $\mathcal{P}=\cup_{n}%
\mathcal{P}_n$. We consider the associated random fields $\xi _{n}$.

Assume $\Omega$ to be periodic (as an embedded set), that is, suppose it has
a finite number of orbits under the action of a rank $d$ group of
translations $\Lambda\cong{\mathbb{Z}}^d$. We also suppose the collection of
patterns ${\mathcal{P}}$ itself to be invariant under $\Lambda $, that is,
it can be written as the disjoint union of the translates of a finite
collection of patterns $\mathcal{P}^{0}={\mathcal{P}}/\Lambda $.

\begin{theorem}
\label{patterncv} 
Let $\Omega$ be a periodic subset of $\R^d$ as above. Let $D\subset\R^d$ be a simply connected domain containing $0$ and define $\Omega_n=D\cap\Omega/n$. We suppose that on each of these is defined a determinantal process with kernel $K_n$ satisfying the conditions in Theorem~\ref{cltdet}. Then, for any periodic set of patterns generated by a pattern $\p^0$, the corresponding field $\xi_n$ satisfies that $n^{-d/2}(\xi _{n}-\mathbb{E}(\xi _{n}))$
converges weakly to a Gaussian white noise with intensity 
\begin{equation}
I(\mathcal{P}^{0})=\frac{1}{|\mathcal{P}^{0}|}\sum_{\mathbf{x}\in \mathcal{P}%
^{0}}\sum_{\mathbf{y}\in {\mathcal{P}}}\mathrm{Cov}(\mathbf{x,y})\,.
\label{intensity}
\end{equation}
\end{theorem}%
\begin{proof}
As explained above, we already know that the limit exists and is Gaussian. Let us identify the covariance structure. In the following, we make the abuse of notation $\varphi(\x)$ to denote the value of $\varphi(z)$ where $z$ is the location (center of mass) of the support of the pattern.

Given any test function $\varphi\in C_0^\infty(D)$, we have
\begin{equation}
\frac{1}{\left\vert {\mathcal{P}}_{n}\right\vert }\mathrm{Var}(\xi
_{n}(\varphi))=\frac{1}{\left\vert {\mathcal{P}}_{n}\right\vert}\sum_{%
\mathbf{x}\in {\mathcal{P}}_{n}}\sum_{\mathbf{y}\in {\mathcal{P}}%
_{n}}\mathrm{Cov}_{n}(%
\mathbf{x,y})\varphi(\x)\varphi(\y)\,.  \label{cov}
\end{equation} 
Consider a small $\delta>0$ (and write $\varepsilon$ for $n^{-d/2}$). We split the previous sum into three contributions as follows
\begin{eqnarray*}
&=&\varepsilon^2\sum_{\x\in \P_n}\varphi(\x)\Bigl(\sum_{\vert \x-\y\vert>\delta}\varphi(\y)\cov_n\left(\x,\y\right)+\sum_{\vert \x-\y\vert\leq\delta}\left(\varphi(\y)-\varphi(\x)\right)\cov_n\left(\x,\y\right)\\
&+&\sum_{\vert \x-\y\vert\leq\delta}\varphi(\x)\cov_n\left(\x,\y\right)\Bigr)\,.
\end{eqnarray*}
We will show that the main contribution comes from the third term.

By the assumptions of Theorem~\ref{cltdet}, we know that there is a matrix $K$ such that for $\x\neq\mathbf{y}$, we have
$$\mathrm{Cov}_{n}\left( \mathbf{x,y}\right) =\det K_{\{\mathbf{x}\}\cup \{\mathbf{y}\}}-\det K_{%
\{\mathbf{x}\}}\det K_{\{\mathbf{y}\}}\,,$$ which is a finite alternating sum of
products that contain at least two off-diagonal block terms. 
Every such
term is bounded by $$O\left( n^{-d}\mathrm{dist}\left( \mathbf{x,y}\right)
^{-d}\right) \,,$$ where $\mathrm{dist}\left( \mathbf{x,y}\right) =\inf_{x\in 
\mathbf{x,}y\in \mathbf{y}}\left\Vert x-y\right\Vert $. Therefore $\cov_{n}\left( \mathbf{x,y}\right) \leq Cn^{-2d}\mathrm{dist}\left( \mathbf{x,y%
}\right) ^{-2d}$. 

The first term therefore yields a contribution of $O\left(\varepsilon^2\varepsilon^{-4}\varepsilon^4/\delta^4\right)=O\left(\varepsilon^2/\delta^4\right)$.
Due to the smoothness of $\varphi$, the second term yields a contribution of 
$$O\left(\varepsilon^2 \delta\sum_{\x,\y\,\vert \x-\y\vert\leq \delta}\cov_n(\x,\y)\right)=O\left(\delta\right)\,,$$
since the sum of covariances is absolutely convergent.

The third term yields a contribution of 
\begin{equation*}\label{thirdterm}
\varepsilon^2\sum_{\x}\varphi(\x)^2\sum_{\y,\vert \y-\x\vert\leq \delta}\cov_n(\x,\y)\,,
\end{equation*}
which, by periodicity, converges, as long as $\delta/\varepsilon\to\infty$, to 
\begin{equation}\label{finalterm}
I(\p)\int_{D}\varphi(z)^2\vert dz\vert^2\,,
\end{equation}
where $I(\p)=\frac{1}{|\mathcal{P}^{0}|%
}\sum_{\mathbf{x}\in \mathcal{P}^{0}}\sum_{\mathbf{y}\in {\mathcal{P}}}%
\mathrm{Cov}(\mathbf{x,y})$. 

By choosing $\delta\to 0$ slowly, for example $\delta=\varepsilon^{1/3}$, the two first contributions vanish and~\eqref{cov} thus converges to~\eqref{finalterm}. 
The process is therefore a Gaussian white noise with intensity $I(\p)$. 
\end{proof}

In the case of singleton patterns, the intensity can be written as 
\begin{equation*}
\frac{1}{|\mathcal{P}^{0}|}\sum_{x\in \mathcal{P}^{0}}\left(K(x,x)-\sum_{y%
\in {\mathcal{P}}}K\left(x,y\right)^{2}\right)\,.
\end{equation*}

\subsection{Interpretation of the intensity}

Suppose the determinantal point field is at the same time a Gibbs random
field, in the sense that the weight of any configuration $T$ is given by $%
\prod_{x\in T}w(x)$, for some $w:\Omega \rightarrow \mathbb{R}^{+}$. In that
case, the noise intensity can be interpreted as a second derivative of a
free energy. Indeed, the partition function is defined by $%
Z=\sum_{T}\prod_{x\in T}w(x)$. In particular, for a fixed pattern $\mathbf{x}%
\subset \Omega $, choose the weight function so that it gives weight $w_{0}w$
on $\mathbf{x}$, and $w$ on all $\mathbf{y}\in \mathcal{P},\mathbf{y}\neq 
\mathbf{x}$. Let $Z\left( w,w_{0}\right) $ be the corresponding partition
function. A short computation shows that $\sum_{\mathbf{y}\in {\mathcal{P}}}%
\mathrm{Cov}(\mathbf{x,y})\,=-\frac{\partial ^{2}}{\partial w\partial w_{0}}%
\log Z\left( w,w_{0}\right) $, which can be interpreted as an electric
susceptibility of the network.

\subsection{Pattern fields of the spanning tree model}

\label{patternfieldsust} Let $D\subset \mathbb{R}^{d}$, $d\geq 2$ be a
bounded simply connected domain containing $0$. Let $\Upsilon$ be an
infinite weighted graph embedded in ${\mathbb{R}}^d$ invariant under a rank $%
d$ lattice $\Lambda\cong {\mathbb{Z}}^d$. We define $\Omega _{n}$ to be the
edge-set of ${\mathcal{G}}_n=\Upsilon/n\cap D$.

The spanning tree model on ${\mathcal{G}}_{n}$ is a symmetric determinantal
process on $\Omega_n$ with kernel given by the transfer current (another
kernel, which is symmetric itself, is the matrix $K$ defined on page~\pageref%
{symkern} in the introduction).

In this context, a pattern $\mathcal{P}^0$ is a finite set of edges $\left\{
e_{1},...,e_{l};e_{l+1},...,e_{k}\right\} $ in a fundamental domain of $%
\Upsilon$, where the edges $\left\{ e_{1},...,e_{l}\right\} $ are present
and $\left\{ e_{l+1},...,e_{k}\right\} $ are forbidden. Suppose that $%
\mathcal{P}^0$ lies inside a fundamental domain $\Upsilon/\Lambda$, and let $%
{\mathcal{P}}_{n}=\Lambda \mathcal{P}^0$ denote the union of its translates
that lie inside $\Omega_n$. We define the corresponding pattern field $\xi
_{n}=\sum_{\mathbf{x}\in {\mathcal{P}}_{n}}\delta _{\mathbf{x}}=\sum_{x\in
\Lambda}\delta _{\mathcal{P}_{x}^0}$, where $\mathcal{P}_{x}^0$ denotes the
translate of pattern $\mathcal{P}^0$ by $x$. 

The mean of the pattern field on a finite set may sometimes be computed:
when $d=2$ and $\Upsilon$ is isoradially embedded, the density of edges has
the following limit. The probability of an edge on an isoradial graph is $\P %
(e)=(2/\pi )\theta _{e}$, where $\theta_e$ is the half-angle of that edge by~%
\cite{Ken02} (see Appendix~\ref{iso}). Hence, the expected number of edges
in a finite set $A$ is 
\begin{equation*}
\mathbb{E}(\xi (A))=\frac{2}{\pi }\sum_{e\in A}\theta _{e}\,.
\end{equation*}

\begin{theorem}\label{USTpattern}
Under the assumptions of Theorem~\ref{ticv}, each rescaled pattern density field $n^{-d/2}(\xi _{n}-\mathbb{E}(\xi _{n}))$ converges weakly in distribution to Gaussian white noise. The sum of two pattern fields also converges to Gaussian white
noise.
\end{theorem}

\begin{proof}
The spanning tree is a determinantal process, hence the result follows from Theorem~\ref{patterncv} provided we have $T\left( x,y\right) \leq O\left( n^{-2d}\left\vert x-y\right\vert
^{-2d}\right) $, which follows from Theorem~\ref{ticv}. 
The correlation between different patterns follows from a similar calculation, and we omit it here.
\end{proof}
The intensity of the white noise is given by~\eqref{intensity}. When we
study correlation among fields (that is, the cross-term in the covariance
structure of the sum of two pattern fields), say between two fields
corresponding to patterns $\mathcal{P}^0$ and $\mathcal{P}^1$, which
generate two collections $\mathcal{P}_n^0$ and $\mathcal{P}_n^1$, the
intensity (which may be negative) is given (computation omitted) by 
\begin{equation*}
I(\mathcal{P}^0,\mathcal{P}^1)=\lim_{n\to\infty}\frac{1}{\sqrt{|\mathcal{P}%
_n^0||\mathcal{P}_n^1|}}\sum_{\mathbf{x}\in \Lambda\mathcal{P}^0_n,\mathbf{y}%
\in\Lambda\mathcal{P}^1_n}\mathrm{Cov}_n(x,y)\,.
\end{equation*}

Let us give the example of an infinite $d$-regular graph: at each vertex,
edges are numbered $1,\ldots,d$. We consider the pattern field $\xi^k$
generated by all edges of type $k$. The intensities of the joint fields $%
\langle\xi^i,\xi^j\rangle$ are denoted $I(i,j)$. %
%
Since the total number of points in $\Omega_n$ is constant ($K_n$ is a
projector), for any $x\in\Omega_n$, we have $0=\mathrm{Cov}%
_n\left(x,\sum_{y\in\Omega_n}y\right)=\sum_{y\in\Omega_n}\mathrm{Cov}_n(x,y)$%
. Hence, we obtain 
\begin{equation}  \label{sumzero}
\sum_{i,j=1}^d I(i,j)=0\,.
\end{equation}
A similar relation is true for the liquid dimers pattern fields in two
dimension, see the remark following Theorem~7 in~\cite{Bo}.

When ${\mathcal{P}}_{n}=\Omega _{n}$ is the collection of all edges, and $%
B\subset D$, the non-rescaled covariance of the variables~$\xi _{n}(B)$
(which represents the number of edges inside $B$) has a limit. (A similar
statement is true for other patterns.)

\begin{proposition}
Under the assumptions of Theorem~\ref{ticv} and given two disjoint subregions $B_{1}$ and $B_{2}$ of $D$, the
correlation between the number of edges $\xi _{n}(B_{1})$ and $%
\xi _{n}(B_{2})$ is 
\begin{equation*}
\mathrm{Cov}\left( \xi _{n}(B_{1}),\xi _{n}(B_{2})\right)
=-\int_{B_{1}\times B_{2}}\left\vert \frac{\partial ^{2}}{\partial z\partial
w}g_{D}(z,w)\right\vert ^{2}|dz|^{2}|dw|^{2}\left( 1+o(1)\right) \,.
\end{equation*}
\end{proposition}
\begin{proof}
This follows from Theorem~\ref{ticv}.
\end{proof}

\subsection{Minimal-pattern fields of the abelian sandpile}

Minimal subconfigurations of the sandpile correspond to local events for the
spanning tree model. However, they cannot always be written as simple
patterns: the probability that a vertex has height zero is equal to a
weighted sum of the probability of (non-disjoint) patterns consisting of a
single edge present (and all other missing). We therefore need to deal with
the mixture of measures and an argument is provided in the following
proposition.

\begin{proposition}\label{mix}
Let $\xi _{n}$ be a random element in $\ell ^{2}(\Omega _{n})$ that
satisfies a central limit theorem (as a sequence in $n$), $\eta _{n}$ be a random measure on finite
subsets of $\Omega _{n}$. For any $\mathbf{x}_{1},...,\mathbf{x}_{n}\subset
\Omega _{n}$, $\eta _{n}\left( \mathbf{x}_{1},...,\mathbf{x}_{n}\right)
=\sum c_{i_{1},...,i_{n}}^{n}\xi _{n}(x_{1}^{\left( i_{1}\right)
},...,x_{n}^{\left( i_{n}\right) })$, where $\{c_{i_{1},...,i_{n}}^{n}\}$
are uniformly bounded, and the sum runs over all the elements of $\mathbf{x}%
_{1},...,\mathbf{x}_{n}$. Then $\eta _{n}$ satisfies a central limit theorem
with the same speed.
\end{proposition}

\begin{proof}
The result follows by directly verifying Wick's formula. Consider the $k^{th}$
moment of the rescaled field $\left( \eta _{n}-\mathbb{E}\eta _{n}\right) /%
\mathrm{Var}\left( \xi _{n}\right) ^{-1/2}$ applied to a test function $\varphi$:
\begin{equation*}
\Xi _{k}^{n}\left( \varphi \right) =\mathrm{Var}\left( \xi _{n}\left(
\varphi \right) \right) ^{-k/2}\sum_{z_{1},...,z_{k}}\varphi \left(
z_{1}\right) ...\varphi \left( z_{k}\right) \mathbb{E}\left(
\prod\limits_{i=1}^{k}\left( \eta _{n}\left( 1_{z_{i}}\right) -\mathbb{E}%
\eta _{n}\left( 1_{z_{i}}\right) \right) \right) .
\end{equation*}%
It suffices to check 
\begin{equation*}
\lim_{n\rightarrow \infty }\Xi _{k}^{n}\left( \varphi \right) =\left\{ 
\begin{array}{cc}
\left( k-1\right) !!\left( \lim_{n\rightarrow \infty }\Xi _{2}^{n}\left(
\varphi \right) \right) ^{k/2} & k\text{ even,} \\ 
0 & k\text{ odd.}%
\end{array}%
\right.
\end{equation*}%
This is easily verified using the corresponding result for $\xi _{n}$. For
instance, when $k$ is even, $\Xi _{k}^{n}\left( \varphi \right) -\left(
k-1\right) !!\left( \Xi _{2}^{n}\left( \varphi \right) \right) ^{k/2}$ only
contains contributions from the sum of $m-$point marginal ($m\geq 3$) of $%
\eta _{n}.$ Each of these marginals can be expressed as a finite linear
combination of $m-$point marginals of $\xi _{n}$. The central limit theorem
for $\xi _{n}$ implies that the corresponding sum of $m-$point marginals of $\xi
_{n}$ is $o\left( 1\right) $. Therefore $\left\vert \Xi _{k}^{n}\left(
\varphi \right) -\left( k-1\right) !!\left( \Xi _{2}^{n}\left( \varphi
\right) \right) ^{k/2}\right\vert= o\left( 1\right)$.
\end{proof}

We now study the abelian sandpile model on the following class of graphs.
Let $\Upsilon $ be an infinite weighted graph embedded in ${\mathbb{R}}^{d}$%
, and invariant under a rank $d$ lattice $\Lambda \cong {\mathbb{Z}}^{d}$.
This graph is in particular amenable and satisfies the one-end property~\cite%
{LMS}. The infinite volume limit on sandpiles is thus well-defined~\cite%
{AJ,JW}.

Let $D\subset {\mathbb{R}}^d$ be a domain. Without loss of generality, we
may suppose that $0$ lies inside the domain $D$ and we define $\Omega _{n}=D
\cap\Upsilon/n$. For a fixed minimal subconfiguration $\mathcal{P}^0$ (that
we call minimal-pattern) lying inside the fundamental domain $%
\Upsilon/\Lambda$, let ${\mathcal{P}}_{n}=\Lambda \mathcal{P}^0\cap \Omega_n$
be the collection of its translates lying inside $\Omega_n$. Let $\xi_{n}$
be the associated random field in~$\Omega_{n}$.

\begin{theorem}\label{ASMpattern}
Under the assumptions of Theorem~\ref{ticv}, and for any minimal-pattern, the corresponding field $n^{-d/2}(\xi_n-\E(\xi_n))$ converges weakly in distribution to Gaussian
white noise.
\end{theorem}

\begin{proof}
When $\Upsilon $ is a regular lattice with uniform weights, each spanning tree\ $\mathcal{T}$
described in Proposition~\ref{minimaldet} carries the same weight. Therefore we can choose a pattern $\p^1$ for the spanning tree model (on the edges of the graph) such that the minimal-pattern is in bijection with the $\p^1$ event under the burning bijection. The field can therefore be seen as a pattern field of the spanning tree model (associated to $\p=\Lambda\p^1$) (note that the union of two adjacent minimal subconfigurations occurs with zero probability, because of Dhar's burning test). The result is then a corollary of Theorem~\ref{USTpattern}.

Let us deal with the general case. We start by showing that the limit exists and is Gaussian. 
Recall from Proposition~\ref{minimaldet} that the probability of a minimal subconfiguration is a linear combination of pattern probabilities for the spanning tree model.
By Proposition~\ref{mix}, it suffices to prove the result for a
fixed tree pattern~$\mathcal{E}$. This follows from Theorem~\ref{USTpattern}.

Let us now identify the covariance structure of the limiting field. As above, it follows from Proposition~\ref{minimaldet} and Theorem~\ref{ticv} that $\cov_n(\x,\mathbf{y})=O\left( n^{-2d}\left\vert x-y\right\vert
^{-2d}\right)$. Hence for any $\x$, $\sum_{\mathbf{y}}\cov_n(\x,\mathbf{y})$ is summable and therefore (as above in the proof of Theorem~\ref{patterncv}) the limiting field is white noise.
\end{proof}

The intensity of the field is computed as follows. The probability of having
two adjacent minimal subconfigurations is zero. The intensity of the white
noise which is of the form~\eqref{intensity} does not include adjacent
patterns. Therefore, each covariance can be replaced (using Proposition~\ref%
{minimaldet}) by a linear combination of expressions involving the transfer
current. It may therefore be evaluated using the explicit formula for the
transfer current when this latter is known.

In the special case of the zero height field, one has $\xi_{n}=\sum_{x\in
\Omega _{n}}\delta _{h_{x}=0}$, where $h_x$ is the height of sand at~$x$.
The asymptotic density of zero height is known in certain cases. Let $%
A\subset \Omega _{n}$ be some set of vertices. We have $\mathbb{E}(\xi
_{n}(A))=|A|/|A/\Lambda |\sum_{x\in A/\Lambda }\P (h_{x}=0)=q|A|/k$, where $%
q=\sum_{x\in A/\Lambda }\P (h_{x}=0)$ and $k=|A/\Lambda |$. In the case of ${%
\mathbb{Z}}^{2}$, the fundamental domain has size $k=1$ and $q=\P %
(h_{0}=0)=2/\pi ^{2}(1-2/\pi )$ by using the explicit expression of the
transfer current (see~\cite{MD} for the computation).

\section{Questions}

\begin{enumerate}
\item We have not dealt with non-minimal subconfigurations in the sandpile:
do the corresponding fields also have Gaussian fluctuations?

\item Can closed-form numerical expressions be obtained for the different
intensities of the pattern fields? Are there algebraic relations between
different intensities, generalizing~\eqref{sumzero}?

\item From the proofs above, one sees that the critical rate of correlation
decay in dimension~$d$ is $r^{-d}$, where~$r$ is the distance. If the
correlation decays faster, the random field fluctuations converge to
Gaussian white noise. Boutillier~\cite{Bo} studied the random field
associated with liquid dimers on planar graphs, which are in the critical
regime, and obtained some long range Gaussian random field in the limit. Are
there any natural critical models in higher dimensions?

\item Does the approximate mean value property hold for supercritical
percolation cluster and random conductance models, thus allowing one to
study the scaling limit of statistical physics models on these random
environments? In the random conductance model on ${\mathbb{Z}}^{d}$~\cite%
{BD10} the Green function is shown to have the same decay $r^{2-d}$ as the
continuous Green function. Provided the approximate mean value property is
valid, one can show that the transfer current has the decay of $r^{-d}$
which implies that the pattern fields considered above (on this random
environment) would have the same fluctuations as in the current paper.

\item What is the distribution of the number of points of a pattern field
for a symmetric determinantal process on a finite space?
\end{enumerate}

\appendix

\section{Isoradial graphs}

\label{iso}

In this appendix, we recall why isoradial graphs are good approximations of
planar domains which satisfy Assumption (A1) and (A2) (on page~\pageref%
{assumption4}) about the convergence of the Green functions on the whole
plane. Our main reference for this section is~\cite{CS}.

\subsection*{Good approximations}

Let us first recall the setting (see e.g.~\cite{Ken02,CS} and references
therein for earlier works on the subject). Let $\Upsilon$ be an infinite
planar isoradial graph and for each $\varepsilon>0$ we denote by $%
\Upsilon^\varepsilon$ a planar isoradial embedding of $\Upsilon$ with mesh
size $\varepsilon>0$ (we may suppose that an isoradial embedding of $\Upsilon
$ with mesh size $1$ is fixed and that other embeddings are obtained by
dilation with respect to a fixed vertex). The planar dual of $\Upsilon$ is
also isoradial with same radius and the \emph{diamond graph} is defined to
be the graph whose faces are the rhombi $R(e)$ for all edges $e$, obtained
by joining the vertices of an edge and its dual. The angle between an edge $e
$ and the side of $R(e)$ at its left is denoted $\theta_e$. It is the
half-angle at the origin of edge $e$ of the rhombus~$R(e)$, see Figure~\ref%
{isoradial}.

\begin{figure}[ht]
\centering
\includegraphics[width=7cm]{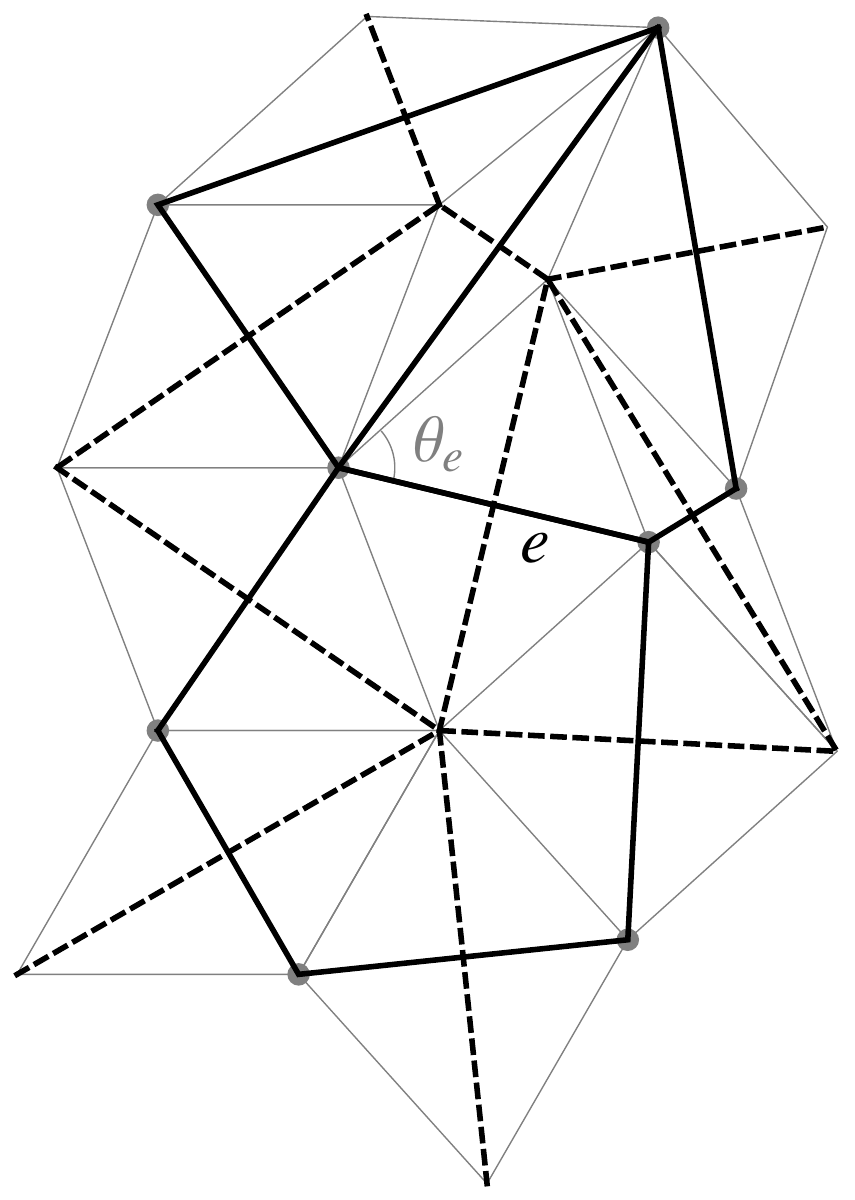}
\caption{A portion of an infinite isoradially embedded graph; The graph is
represented in thick lines; the dashed edges represent the dual graph; the
gray lines form the diamond graph; An edge $e$ of the graph and the
corresponding rhombic half-angle~$\protect\theta_e$ is represented.}
\label{isoradial}
\end{figure}

We make the customary assumption (bounded half-angles property) that all
angles $\theta_e$ of these rhombi are uniformly bounded away from $0$ and $%
\pi/2$. This ensures that for vertices of $\Upsilon^\varepsilon$, the
combinatorial distance in the graph and the Euclidean distance are uniformly
related by a factor of $\varepsilon$. Let $D\subset\mathbb{C}$ be a
simply-connected compact planar domain. For every $\varepsilon>0$, consider $%
D^\varepsilon$ to be a finite isoradial subgraph of~$\Upsilon^\varepsilon$
whose vertices lie in $D\cap\Upsilon^\varepsilon$, which is ``simply
connected'' in the sense that the union of its closed faces is simply
connected. 
A sequence $\left(D^{\varepsilon}\right) _{\varepsilon >0}$ is said to \emph{%
approximate} $D$ if the Hausdorff distance from $D^{\varepsilon}$ to $D$ is $%
O(\varepsilon)$.

In the following, we assume that $D^\varepsilon$ is endowed with its \emph{%
critical isoradial conductances} 
\begin{equation*}
c(e)=\tan\theta_e\,,
\end{equation*}
where $\theta_e$ is the half-angle of the rhombi $R(e)$ (see~\cite{Ken02}
for a longer definition), and consider the corresponding Laplacian~$\Delta$.

It is shown in~\cite{CS} that isoradial graphs satisfy the \emph{Approximate
mean value} property. The \emph{Paths approximation} property follows from
the bounded angle property mentioned above. Thus, isoradial graphs are good
approximations.

Furthemore, the random walk (biased by weights) on the embedded graph is
isotropic. A local time-reparametrization of it converges to Brownian
motion. This implies assumption~(A2).

\subsection*{Green's functions convergence}

It was proved by Kenyon in~\cite{Ken02}, up to an improvement given by B\"{u}%
cking in~\cite{Bu}, that the Green function on $\Upsilon^\varepsilon$
satisfies the following expansion~\footnote{%
On $\mathbb{Z}^2$, an all-order expansion is known~\cite{FU}.}: for any two
vertices $v\neq w$, we have 
\begin{equation*}
G(v,w)=-\frac{1}{2\pi}\log\vert v-w\vert+O\left(\frac{\epsilon^2}{\vert
v-w\vert^2}\right)\,.
\end{equation*}

This implies that $G$ converges to the Green function on the whole plane $g_%
\mathbb{C}(z,w)=-\frac{1}{2\pi}\log|z-w|$, as $\varepsilon\to 0$.

This is assumption (A1). We have thus recalled why isoradial graphs are good
approximations satisfying assumptions (A1) and (A2).

Since the Laplace equation is conformally invariant, for any
simply-connected surface with boundary, $D$, the Neumann and Dirichlet Green
function are obtained by their image under a conformal map between the upper
half plane $\H$ and $D$ of the Green function on $\H$ with corresponding
boundary conditions.

Discrete harmonic functions on isoradial graphs converge to continuous
harmonic functions in a very strong sense described in~\cite{CS}. In
particular, the Dirichlet Green function is shown to converge building on
Kenyon's asymptotics for the whole plane. We give here the proof
for the free boundary Green function. The proof follows from arguments in~%
\cite{CS} although the result is not stated explictly there.

\begin{proposition}\label{surface}
For any graph $\G$ approximating a simply connected surface with boundary $D$ in the sense that $\G$ is embedded in $D$ and there exists a system of coordinate patches and conformal maps that map the graph to an isoradial planar graph on these patches, the discrete Neumann Green function converges uniformly on any compact set inside the surface to its continuous counterpart. 
\end{proposition}
\begin{proof}
Let us first show it for the unit disk $D$. Let $\Upsilon^\varepsilon$ be an infinite isoradial graph with mesh size $\varepsilon$ such that $\G$ is a subgraph of it. Write $G=G_{\Upsilon^\varepsilon}-H$ where $H$ is harmonic with Neumann boundary conditions equal to the normal derivative of $G_{\Upsilon^\varepsilon}$. By $C^1$ convergence of $G_{\Upsilon^\varepsilon}$ when the mesh size goes to zero, the values of the normal derivative converge. Now consider the dual graph of $\G$. The harmonic conjugate~$H^*$ of $H$ is univalued since $D$ is simply-connected. Its values on the boundary are determined (up to a constant which we take to be $0$) and converge to a limit $f^*$. By~\cite{CS} the harmonic function $H^*$ converges to the harmonic extension of $f^*$. Hence, its dual $H$ converges too, to a function $h$. The limit of~$G$ is therefore $g_\C-h$ which is equal to Neumann Green's function $g_D^r$ (up to a constant).

In the case where $D$ is another domain (even non planar), we use a conformal map to bring it back to the previous case. Since convergence is a local result, we may again use the convergence result of~\cite{CS} and the above argument.
\end{proof}

\subsection*{Comments on previous related work}

Asymptotics of the Green function on isoradial graphs and its derivatives
have been well studied in the literature. Among others, we may state~\cite%
{St,FU} for asymptotics of the Green function on $\mathbb{Z}^2$and \cite%
{Ken02,Bu} for the case of isoradial graphs on the whole plane (stated
above). For the convergence of the increment rate of the rescaled Green
function over $\varepsilon \mathbb{Z}^2$, see Lemma 17 in~\cite{Ken00}.

D\"{u}rre in~\cite{Du} showed in the case of ${\mathbb{Z}}^2$in the
Dirichlet case the analog of our Theorem~\ref{ticv} and also gave the
application to the study of zero-height fields in the abelian sandpile (our
Theorems~\ref{corrASM} and~\ref{ASMpattern} for this particular
minimal-pattern in the case of ${\mathbb{Z}}^2$).

Important results of convergence of discrete harmonic functions and their
derivatives (in the Carath\'eodory topology) in domains with boundary (in
particular the Dirichlet Green function) to their continuous counterpart are
gathered in~\cite{CS}. Theorem~\ref{ticv} directly follows from their
arguments in the case of isoradial graphs by using $C^1$convergence of the
wired Green function in each variable. For the free case it follows from
Proposition~\ref{surface} above.

Carath\'eodory convergence is weaker than Hausdorff convergence, which we
suppose. Therefore, the results of~\cite{CS} applied to the transfer current
give the uniform~$C^1$convergence.

Note that for isoradial graphs there is an explicit formula for the transfer
current in terms of path integral which is a linear combination of the Green
function explicit formula derived in~\cite{Ken02}. We do not use it here
because we are interested only in the scaling limit asymptotics which are
easier to derive.

\addtocontents{toc}{\SkipTocEntry}
\section*{Acknowledgements}

We thank C\'edric Boutillier, Richard Kenyon, and David Wilson
for helpful discussions and feedback. We also thank the anonymous referee for helpful suggestions.


\end{document}